\documentclass[epsfig,leqno,12pt]{amsart}

 \usepackage{eepic,epic}            

\theoremstyle{plain}
\newtheorem*{theo}{Theorem}
\newtheorem{asse}{Assertion}

\newtheorem*{coro}{Corollary}

\theoremstyle{remark}

\newtheorem*{rema}{Remark}

\newtheorem*{defn}{Definition}

\def\U{{\frak U}}

\def\Z{{\mathbb Z}}

\def\C{{\mathbb C}}
\def\R{{\mathbb R}}

\def\al{\alpha}

\def\gcd{{\operatorname{gcd}}}

\def\id{{\operatorname{id}}}
\def\rank{{\operatorname{rank}}}

\def\db[#1\db]{%
 \setbox0=\hbox{$#1$}\argwidth=\wd0
 \setbox0=\hbox{$\left[\box0\right]$}
    \advance\argwidth by -\wd0
 \left[\kern.3\argwidth\box0 \kern.3\argwidth\right]}

\numberwithin{equation}{subsection}

\title
[Opposite power Series]
{Opposite Power  Series }

\author
{
Kyoji Saito
}
\address{ IPMU, university of Tokyo}
 \thanks{}
 \centerline{} 

\begin{document}

\maketitle
\begin{center}
\textit{ Dedicated to Professor Antonio Mach\`i\\ on the occasion of
 his 70th birthday.
}
\end{center}

 \begin{abstract}
In order to analyze the singularities of a power series function $P(t)$ on the
  boundary of its convergent disc, where $P(t)$ was mainly the
  growth function (Poincar\'e series) for a finitely generated group or
  a monoid \cite{S1}, we introduced the
  space $\Omega(P)$ of {\it opposite power series} in the opposite variable
  $s\!=\!1/t$.
 In the present 
  paper, forgetting about the geometric or combinatorial background on $P(t)$,
  we study the space $\Omega(P)$ abstractly for any suitably tame power series
  $P(t)\!\in\!\C\{t\}$. For the case when $\Omega(P)$ is a finite set and
  $P(t)$ is  meromorphic in a neighbourhood of the closure of its
  convergent disc, 
  we show {\it a duality between $\Omega(P)$ and the
  highest order poles of $P(t)$ on the boundary of its
  convergent disc.} 
 \end{abstract} 
\tableofcontents
\newpage

\section{Introduction}\label{sec:1}

There seems a remarkable ``resonance'' between oscillation
behavior\footnote{By an oscillation behavior, we mean that, for each
fixed $k\!\in\!\Z_\ge0$ called a period, the
sequence of the rate $\gamma_{n\!-\!k}\!/\!\gamma_n$ ($n\!\in\!\Z_{\!>\!\!>\!0}$)
has several different accumulation values.}
of a sequence $\{\gamma_n\}_{n\in\Z_{\ge0}}$ of complex numbers satisfying a
tame condition (see equation \eqref{eq:2.1.2}) and the singularities of
its generating function $P(t)\!=\!\sum_{n=0}^\infty \gamma_n t^n$ on the
boundary 
of the disc of convergence in $\C$.\ The idea was
inspired by and strongly used in the study of growth
functions (Poincar\'e series) for finitely
generated groups and monoids \cite[\S11]{S1}. 

Let us explain the ``resonance'' by a typical example due to Mach\`i
\cite{M} (for details, see Examples in \S3.3 and \S5.4 of the
present paper. Other simple examples are given in \S3.4 (see \cite{C,S2,S3})
and \S3.5).
By choosing generators of order 2 and 3 in $\mathrm{PSL}(2,\Z)$, Mach\`i has
shown that the number $\gamma_n$ of elements of $\mathrm{PSL}(2,\Z)$ which are
expressed in words of length less or equal than $n\!\in\!\Z_{\ge\!0}$ w.r.t.\ the
generators is given by $\gamma_{2k}\!=\! 7\cdot \!2^k\!-\!6$ and 
$\gamma_{2k+1}\! =\! 10\cdot\!2^k\!-\!6$ for $k\!\in\!\Z_{\ge0}$.\
On one hand, this means that the sequence of ratios $\gamma_{n-1}/\gamma_n$
{\small ($n\!=\!1,2,\cdots$)}
accumulates to {\it two distinct ``oscillation'' values}  {\small \{$\frac{5}{7},\frac{7}{10}$\}} according as $n$ is even or odd. 
On the other hand, the generating function (or, so called, the growth function) 
 \vspace{-0.1cm}
can be expressed as a
rational function {\small $P(t)$}$=\!\!
\frac{(1+t)(1+2t)}{(1-2t^2)(1-t)}$, and it has {\it two poles} at
 \vspace{-0.1cm}
{\small \{$\pm\frac{1}{\sqrt{2}}$\}} on the boundary of its convergent disc
of radius\! {\small $\frac{1}{\sqrt{2}}$}.\ 
 \vspace{-0.1cm} 
{\it We see that  there is a ``resonance'' between 
the set {\small \{$\frac{5}{7},\frac{7}{10}$\}} of
 \vspace{-0.1cm}
``oscillations'' of the sequence $\{\gamma_n\}_{n\in\Z_{\ge0}}$
and the set {\small \{$\pm\frac{1}{\sqrt{2}}$\}} of ``poles'' of the function
$P(t)$,} in the way we shall explain in the present paper.

In order to analyze these phenomena, in \cite[\S11]{S1},
we introduced a 
space $\Omega(P)$ of {\it opposite power series} in the opposite variable
$s\!=\!1/t$, as a compact subset of $\C[[s]]$, where each opposite
series is defined by using 
``oscillations'' of the sequence $\{\gamma_n\}_{n\in\Z_{\ge0}}$ so that
$\Omega(P)$ carries a comprehensive information of oscillations (see
\S2.2 Definition \eqref{eq:2.2.2}).
 On the other hand,  the
space $\Omega(P)$ has duality with the singularities of the function
$P(t)$ (\S5 Theorem). Thus, $\Omega(P)$ becomes a bridge between the two
subjects: oscillations of $\{\gamma_n\}_{n\in\Z_{\ge0}}$ and singularities
of $P(t)$. 
%
%
%
Since the method is independent of the group theoretic
background and is extendable to a wider class of series (see \S2.1 Example 2), which we call {\it
tame}, 
we separate the 
results and proofs in a self-contained way in the present
paper.
We study in details the case when $\Omega(P)$ is finite, where we have
good understanding of the above mentioned resonance by a use of {\it rational subset}
explained in the following paragraph, and Mach\`i's example is understood in that frame.  

One key concept in the present paper is a {\it rational subset} $U$ (\S3), which is a subset of the positive integers $\Z_{\ge0}$ such
that 
the sum $\sum_{n\in U}t^n$ is a rational function in $t$ (i.e.\ $U$, up to finite, is a finite union of arithmetic progressions). The
concept is used twice in the present paper.
The first time it is used is in \S3, where we show that, if the space of opposite series $\Omega(P)$ is finite, then there is a finite partition
$\Z_{\ge0}=\amalg_{i}U_i$ of $\Z_{\ge0}$ into rational subsets so that
there is no longer oscillation inside in each $\{\gamma_n: n\in U_i\}$. We call such phenomena ``finite
rational accumulation'' (\S3.2 Theorem) (such phenomena already appeared when we were
studying the F-limit functions for monoids \cite[\S11.5 Lemma]{S1}). The
second time it is used is in \S5, where we introduce
a rational operator $T_U$ acting on a power series $P(t)\in\C[[t]]$ by
letting $T_{U}P(t)\!:=\!\sum_{n\in U} \gamma_nt^n$. The rational operators
form a machine that ``manipulates'' singularities of the power series $P(t)$. 
In this way, rational subsets combine the oscillation of
a sequence $\{\gamma_n\}_{n\in\Z_{\ge0}}$ and the singularities of the
generating function $P(t)\!:=\!\sum_{n=0}^\infty \gamma_nt^n$ for the case
when $\Omega(P)$ is finite.

\medskip
The contents  of the present paper are as follows.

In \S2, we introduce the space $\Omega(P)$ of opposite series as 
the accumulating subset in $\C[[s]]$  of the sequence
$X_n(P)\! :=\!\sum_{k=0}^n \frac{\gamma_{n-k}}{\gamma_n}s^k\
(n\!=\!0,1,2,\cdots)$ with respect to the coefficient-wise convergence
topology, where the $k$th coefficient describes an oscillation of
period $k$. 
Dividing by period-one oscillation, 
we construct a shift action $\tau_\Omega$ on 
the set $\Omega(P)$ to itself, which shifts $k$-period oscillations to
$k-1$-period oscillations. 

In 3.1, we introduce the key concept: {\it finite rational accumulation}. 
We show that if $\Omega(P)$ is    
a finite set, then $\Omega(P)$ is automatically
a finite rational accumulation set and the $\tau_\Omega$-action becomes 
invertible and transitive. That is, $\tau_\Omega$ is acting cyclically
on $\Omega(P)$.

Starting with \S4, we assume always 
finite rational accumulation for $\Omega(P)$.
In \S4, we analyze in details of the opposite series in $\Omega(P)$ and
the module $\C\Omega(P)$ spanned by $\Omega(P)$,
showing that the opposite series become rational functions with the common denominator
$\Delta^{op}(s)$ in 4.1, and that the rank of $\C\Omega(P)$ is equal to
$\deg(\Delta^{op}(s))$ in \S4.4.

In \S5, we assume that the series $P(t)$ defines a meromorphic 
function in a neighbourhood of the closed convergent disc. 
Then we show that $\Delta^{op}(s)$ is opposite to the polynomial $\Delta^{top}(t)$ of the highest order
part of poles of $P(t)$ (Duality 
Theorem in \S5.3), and, in particular,  the rank of the 
space $\C\Omega(P)$ is equal to the number of poles 
of the highest order of $P(t)$ on the boundary of the convergent disc. 
We get an identification of some transition matrices obtained in  $s$-side and
in $t$-side, which plays a crucial role in the trace formula for limit
F-function \cite[11.5.6]{S1}.

\medskip
\noindent
{\bf  Problems.}
The space $\Omega(P)$ is new with respect to the study of the
singularities of a power series function $P(t)$, and the author thinks the following
directions of further study may be rewarding. 

\smallskip
\noindent
1. 
Generalize the space  $\Omega(P)$ in order to capture lower order poles
of 

$P(t)$ on the boundary of its convergent disc (c.f.\  \cite[\S12, {\bf 2.}]{S1}).\!

\noindent
2. Generalize the duality for the case when $\Omega(P)$
is infinite. Some 

probabilistic approach may be desirable (c.f.\  \cite[\S12, {\bf 1.}]{S1}).\!


\section{ The space of opposite series.}

In this section, we introduce the space $\Omega(P)$ of opposite series
for a tame power series $P\in\C[[t]]$, and equip it with a
$\tau_\Omega$-action.

\subsection{Tame power series}\label{subsec:2.1}\hspace{0cm}

Let us call a complex coefficient power series in $t$
\begin{equation}
\label{eq:2.1.1}
\begin{array}{l}
P(t) \ =\ \sum_{n=0}^\infty \gamma_nt^n
\end{array}
\end{equation}
to be {\it tame}, if there are positive real numbers $u,v\in\R_{>0}$ such that 
\begin{equation}
\label{eq:2.1.2}
 u\ \le\ |\gamma_{n-1}/\gamma_n|\ \le\ v
\end{equation}
for sufficiently large integers $n$ (i.e.\ for $n\ge N_P$ for some $N_P\in\Z_{\ge0}$).
This implies that there are positive constants $c_1,c_2$ with $c_1\!\le\!c_2$ so that
\begin{equation}
\label{eq:2.1.3}
c_1v^{-n}\le |\gamma_n| \le c_2u^{-n}
\end{equation}
for sufficiently large integer $n\in\Z_{\ge0}$ (actually, put
$c_1\!=\!|\gamma_{N_P}|v^{N_P}$ and $c_2\!=\!|\gamma_{N_P}|u^{N_P}$ for
$n\ge N_P$).
Let us consider two limit values:
\begin{equation}
\label{eq:2.1.4}
 u\ \le\
 r_P:=\!1/\underset{n\to\infty}{\overline{\lim}}|\gamma_n|^{1/n}\ \le\ R_P:=\!1/\underset{n\to\infty}{\underline{\lim}}|\gamma_n|^{1/n}\ \le\ v.
\end{equation}
Cauchy-Hadamard Theorem says that $P$ is convergent of radius $r_P$. 
 
\medskip
\noindent
{\bf Example 1.} Let $\Gamma$ be a group or a monoid with a finite
generator system $G$.
Then the length $l(g)$ of an element $g\in \Gamma$ is the shortest length of
words expressing $g$ in the letter $G$. Set $\Gamma_n:=\{g\in\Gamma\mid 
l(g)\le n\}$ and $\gamma_n:=\#(\Gamma_n)$. Then 
the growth function (Poincar\'e series) for $\Gamma$ with respect to $G$ is defined by 
$P_{\Gamma,G}(t):=\sum_{n=0}^\infty \gamma_nt^n$.
The sequence  
$\{\gamma_n\}_{n\in\Z_{\ge0}}$ is increasing and semi-multiplicative
$\gamma_{m+n}\!\le\!\gamma_m\gamma_n$. Therefore, by choosing 
 $u\!=\!1/\gamma_1$ and $v\!=\!1$, the growth series is tame. 

{\bf 2.} Ramsey's theorem says that, for any $n\in\Z_{>0}$, there
exists a positive integer $N$ such that if the edges of the complete
graph on $N$ vertices are colored either red or blue, then there exists
$n$ vertices such that all edges joining them have the same colour. The
least such integer $N$ is denoted by $R(n)$, 
and is called the $n$th {\it diagonal Ramsey number},
e.g.\
$R(1)\!=\!1,R(2)\!=\!2,R(3)\!=\!6,R(4)\!=\!18$ (c.f.\ \cite{SR}). Then, the following
estimates are known due to Erd\"os \cite{E} and Szekeres:
\[
 2^{n/2}\le R(n)\le 2^{2n}.
\] 
So, $R(t):=\sum_{n=0}^\infty R(n)t^n$ (where put $R(0)\!=\!1$) form
a tame series.

\subsection{The space  $\Omega(P)$ of opposite series}
\label{2.2}\hspace{0cm}

Let $P$ be a tame power series. Then, there is a positive integer $N_P$
such that $\gamma_n$ is invertible for all $n\ge N_P$. Therefore, for
$n\in\Z_{\ge N_P}$, we define the {\it opposite polynomial of degree} $n$ by 
\begin{equation}
\label{eq:2.2.1}
\begin{array}{ll}
X_n(P)\ := \ \sum_{k=0}^n\frac{\gamma_{n-k}}{\gamma_n}\ s^k. 
\end{array}
\end{equation}
Regarding $\{X_n(P)\}_{n\ge N_P}$ as a sequence in the space $\C[[s]]$
of formal power series, where $\C[[s]]$ is equipped with the 
classical topology, i.e.\ the product topology of coefficient-wise
 convergence in classical topology, we define {\it the space of
 opposite series} by 
\begin{eqnarray}
\label{eq:2.2.2}
&\qquad \Omega(P)  := %
\substack{
\text{\normalsize the set of accumulation points of 
the sequence} \\ 
\text{\normalsize \eqref{eq:2.2.1} with respect to the classical topology. }
}
\end{eqnarray} 
That is, an element of $\Omega(P)$ can be viewed as an equivalence class
of infinite convergent subsequences $\{X_{n_m}(P)\}_m$ of opposite polynomials.

\smallskip
The first statement on $\Omega(P)$ is the following.

\begin{asse} Let $P$ be a tame series. Then  
$\Omega(P)$ is a non-empty compact closed subset of $\C[[s]]$.
\end{asse}
\begin{proof}
For each $k\!\in\!\Z_{\ge0}$, the $k$th coefficient $\frac{\gamma_{n-k}}{\gamma_n}$ of the polynomial $X_n(P)$ for
 sufficiently large $n\!\in\!\Z_{\ge0}$ with respect to $P$ and $k$ (i.e.\ for $n\ge N_P+k-1$) has the
 approximation $u^k\le
 |\frac{\gamma_{n-k}}{\gamma_n}|\!=\!|\frac{\gamma_{n-1}}{\gamma_n}||\frac{\gamma_{n-2}}{\gamma_{n-1}}|\cdots|\frac{\gamma_{n-k}}{\gamma_{n-k+1}}|\le
 v^k$, i.e.\ it lies in the compact annulus
\[
 \bar{D}(0,u^k,v^k):=\{a\!\in\!\C\mid u^k\!\le\! |a|\!\le\! v^k\}.  
\]
Thus, for each fixed $m\!\in\!\Z_{\ge0}$, the image of the sequence
 \eqref{eq:2.2.1} under the truncation map
$\pi_{\le m}:\C[[s]]\to \C^{m+1},\ \sum_{k=0}^\infty a_ks^k\mapsto
 (a_0,\cdots,a_m)$ 
accumulates to an non-empty compact subset of
 $\prod_{k=0}^m\bar{D}(0,u^k,v^k)$, say $\Omega_{\le m}$.
Then, we have:
\[\begin{array}{ll}
 \Omega(P) 
=\cap_{m=0}^\infty \big((\pi_{\le m})^{-1}\Omega_{\le m}\cap
 \prod_{k=0}^\infty\bar{D}(0,u^k,v^k)\big),
\end{array}
\]
where the RHS, as an intersection of decreasing sequence of compact sets, is
 non-empty and compact.
\end{proof}

An element $a(s)\!=\!\Sigma_{k=0}^\infty a_ks^k$ of $\Omega(P)$ 
is called an {\it opposite series}.
Its $k$th coefficients $a_k$, i.e.\ an {\it oscillation value of period} $k$,
belongs to $\bar{D}(0,u^k,v^k)$. Given an opposite series
$a(s)$, the constant
term $a_0$ is equal to 1. The coefficient $a_1$, i.e.\ oscillation value
of period 1, is  called the {\it initial} of the opposite series $a$, and denoted by $\iota(a)$.

For later use, let us introduce an auxiliary space of the initials:
\begin{equation}
\label{eq:2.2.3}
\Omega_1(P) := \text{the accumulation set of the sequence
$\Big\{\!\frac{\gamma_{n-1}}{\gamma_n}\!\Big\}_{\!n\gg 0}$}, 
\!\!\!\!\!\!
\end{equation}
which is a compact subset in $\bar{D}(0,u, v)$. 
The projection map 
$\Omega(P)\to\Omega_1(P),\ a\mapsto \iota(a)$ 
is surjective but may not be injective (see \S3.5 Ex.).

\subsection{The $\tau_\Omega$-action on $\Omega(P)$}\hspace{0cm}
\label{subsec:2.3}

We introduce a continuous map $\tau_\Omega$ form $\Omega(P)$ to itself.

\begin{asse} 
{\bf a.\ }Let $\{n_m\}_{m\in\Z_{\ge0}}$ be a subsequence of $\Z_{\ge0}$ tending to
 $\infty$. If the sequence $\{X_{n_m}(P)\}_{m\in\Z_{\ge0}}$ converges 
to an opposite series $a$, then the sequence $\{X_{n_m-1}(P)\}_{m\in\Z_{\ge0}}$
also converges to an opposite series, whose limit depends only on $a$
 and is denoted by $\tau_{\Omega}(a)$. 
Then, we have 
\begin{equation}
\label{eq:2.3.1}
 \tau_\Omega(a)\ =\ (a-1)/\iota(a) s.
\end{equation}

{\bf b.\ }Let $\C\Omega(P)$ 
be the $\C$-linear subspace of $\C[[s]]$ spanned by 
$\Omega(P)$. Then the map 
$\tau:\Omega(P)\longrightarrow \C\Omega(P), 
\quad  a  \mapsto  \iota(a) \tau_\Omega(a) 
$
\noindent
naturally extends to an endomorphism of
$\C\Omega(P)$.
\begin{equation}
\label{2.3.3}
\tau\in End_\C(\C\Omega(P))
\end{equation}
\end{asse}

\begin{proof} a.\ By definition, for any $k\in\Z_{\ge0}$, the sequence $\frac{\gamma_{n_m\!-\!k}}{\gamma_{n_m}}$
converges to a constant $a_k\in \bar{D}(u^k,v^k)$. Then,
 $\frac{\gamma_{(n_m\!-\!1)-(k-1)}}{\gamma_{n_m\!-\!1}}=\frac{\gamma_{n_m\!-\!k}}{\gamma_{n_m}}/\frac{\gamma_{n_m\!-\!1}}{\gamma_{n_m}}$
 converges to $a_k/a_1$. That is, the sequence
 $\{X_{n_m-1}(P)\}_{m\in\Z_{\ge0}}$ converges to an opposite series,
 whose $(k\!-\!1)$th coefficient is equal to $a_k/a_1$.

b. This is trivial, since $ a \mapsto \iota(a) \tau_\Omega(a)$ is a
 restriction on $\Omega(P)$ of an affine linear endomorphism $(a-1)/s$ on $\C[[s]]$.
\end{proof}

\subsection{Examples of $\tau_\Omega$-actions} \hspace{0cm}

At present, except for the trivial cases when $\#\Omega(P)\!=\!1$ so that $\tau_\Omega\!=\!\id$, there are only few
examples where the action  $(\Omega(P_{\Gamma,G}),\tau_\Omega)$ is
explicitly known: namely, the groups of the form
$\Gamma\!=\!(\Z/p_1\Z)\!*\cdots*\! \Z/p_n\Z$ for some
$p_1,\cdots,p_n\!\in\!\Z_{>1}$  ($n\!\ge\!2$) with
the generator system $G\!=\!\{a_1,\cdots,a_n\}$ where $a_i$ is the standard
generator of $\Z/p_i\Z$ for $1\!\le\! i\!\le\! n$, which
include Mach\`i's example (see \S3.3-4). 

For the tame series $R(t)$ in \S2.1 Example 2, we know nothing
 about $(\Omega(R),\tau_\Omega)$. It is already a question whether
 $\#\Omega(R)$ is equal to $1$, finite many $(>\!1)$, or infinite? The author would like to
 expect  $\#\Omega(R)\!=\!1$.

\subsection{Stability of $\Omega(P)$}\hspace{0cm}

\label{subsec:2.4}

In the present subsection, we are (mainly) concerned with following type
of questions, which we will call {\it stability questions concerning
$\Omega(P)$}: 
for a given tame series $P$, under which assumptions on another power
series $Q$, is $P+Q$  again tame and $\Omega(P)\!=\!\Omega(P+Q)$? Or, if 
$\Omega(P+Q)$ changes from $\Omega(P)$, how does it change? 

We discuss some miscellaneous results related to stability questions, but we
do not pursue full generalities. Except that
Assertion 3 is used in the proof of Assertion 13, results in the present
paragraph are not used in the
present article. Therefore, the reader may choose to skip the part of
this subsection after Assertion 3 without substantial loss.

\begin{asse}
{\it Let $Q\!=\!\sum_{n=0}^\infty q_nt^n$ converge in the disc of radius $r_Q$ such that $r_Q>R_P$. Then $P+Q$ is tame and $\Omega(P)=\Omega(P+Q)$.}
\end{asse}

\noindent
{\it Proof.} Let $c$ be a real number satisfying
 $r_Q\!>c\!>\!R_P$.
 Then, one has $\underset{n\to\infty}{\lim}
 q_nc^n\!=\!0$ and $c^n\!\ge\! 1/|\gamma_n|$ for sufficiently large $n$.
 This implies $\underset{n\to\infty}{\lim}
 \frac{\gamma_n+q_n}{\gamma_n}\!=\!1\!+\!\underset{n\to\infty}{\lim}\frac{q_n}{\gamma_n}\!=\!1$. The
 required properties follow. \quad $\Box$ 

\smallskip
\begin{asse} {\it Let $r$ be a positive real number with $r\!<\!R_P$. If
 $\Omega_1(P)\cap\{z\!\in\!\C : |z|\!=\!r\}\!=\!\emptyset$. Then there exists a
 power series $Q(t)$ of radius of convergence $r_Q\!=\!r$ such that $P\!+\!Q$ is tame and $\Omega(P\!+\!Q)\not\subset \Omega(P)$.
}
\end{asse}
\begin{proof}  We define the coefficients of
 $Q(t)=\sum_{n=0}^\infty q_nt^n$ by the following conditions:
 $|q_n|=r^{-n}$ and $\arg(q_n)=\arg(\gamma_n)$. Then, for tameness of
 $P+Q$, we have to show some positive bounds $0\!<\!U\!\le\!A_n\!\le\!
 V$ for $A_n\!=\!|\frac{\gamma_{n-1}+q_{n-1}}{\gamma_n+q_n}|$.  Since $|\gamma_n\!+\!q_n|\!=\!|\gamma_n|\!+\!r^{-n}$, we have 
$A_n\!=\!\frac{|\gamma_{n-1}/\gamma_n|+r/(|\gamma_n|r^n)}{1+1/(|\gamma_n|r^n)}$. 
Then, evaluating term-by-term in the numerator, one gets
 $A_n\!\le\!v\!+\!r\!=:V$. On the other hand, according as $1\ge
 1/(|\gamma_n|r^n)$ or not, we have $A_n\ge u/2$ or $A_n\ge r/2$. 
Therefore, we may set $U\!:=\!\min\{u/2,r/2\}$.

Let us find a particular element $d\in \Omega(P+Q)$ such
 that $d\not\in \Omega(P)$. For a small positive real number
 $\varepsilon$ satisfying the inequality $(1\!-\!\varepsilon)/r\!
 >\!1/R_P$, there exists an
 increasing infinite sequence of integers $n_m$ ($m\!\in\!\Z_{\ge0}$) such
 that 
 $((1\!-\!\varepsilon)/r)^{n_m}\!>\!|\gamma_{n_m}|$ for $m\!\in\!\Z_{\ge0}$. 
By choosing a suitable sub-sequence (denoted by the same $n_m$), we may
 assume that 
 $X_{n_m}(P\!+\!Q)$ converges to an element, say $d$, in $\Omega(P+Q)$. Its $k$th
 coefficient $d_k$ is equal to the limit of the sequence 
$(\gamma_{n_m\!-\!k}\!+\!q_{n_m\!-\!k})/(\gamma_{n_m}\!+\!q_{n_m})$ for
 $n_m\!\to\!\infty$. For each fixed $n_m$, dividing
 the numerator and the denominator by $q_{n_m}$, we get an expression $(
 X\!+\!r^{k}Y)/(Z\!+\!1)$ where 
$|X|=|\gamma_{n_m-k}/\gamma_{n_m}|\cdot|\gamma_{n_m}r^{n_m}|\le
 v^k\cdot (1\!-\!\varepsilon)^{n_m}$ (for $n>>k$), $Y\in S^1$, and 
 $|Z|=|\gamma_{n_m}r^{n_m}|<(1\!-\!\varepsilon)^{n_m}$. Thus, taking
 the limit $n_m\to\infty$, we have $X\to 0$, $Y\to e^{i\theta_k}$ for
some $\theta_k\in\R$ and $Z\to0$ so that $d_k\!=\!r^{k}e^{i\theta_k}$. 
On the other hand, we see that $d\!\not \in\! \Omega(P)$, since
 $\iota(d)\!=\!r e^{i\theta_1}\not\!\in\!\Omega_1(P)$ by assumption.
\end{proof}

We do not use following Assertion in the present paper, since we know
more precise information for the cases
$\#\Omega(P)\!<\!\infty$. However, it may have a significance  when we
study the general case with $\#\Omega(P)\!=\!\infty$.

\begin{asse}\! An opposite series converges with radius
 $1/\sup\{|a|: a\!\in\!\Omega_1(P)\}\le 1/R_P$.
\end{asse}
 \begin{proof}

Let  $a(s)=\underset{m\to\infty}\lim X_{n_m}(P)$ for an increasing
  sequence $\{n_m\}_{m\in\Z_{\ge0}}$ be an opposite series. By the Cauchy-Hadmard theorem, the
  radius of convergence of $a$ is given by
{\small \[
 r_a = 1/\underset{k\to\infty}{\overline{\lim}}|a_k|^{1/k}   =
  1/\underset{k\to\infty}{\overline{\lim}}\ |\underset{m\to\infty}\lim\gamma_{n_m-k}/\gamma_{n_m}
   |^{1/k},
\]
where the RHS is lower bounded by $1/\sup\{|a| :  a\in\Omega_1(P)\}$ from below.
}
\end{proof}
\noindent
It seems natural to ask when we can replace $\sup\{|a| :  a\in\Omega_1(P)\}$ by $R_P$?
Finally, we state a result, which is not related to the stability.

\begin{asse}
For any positive integer $m$, we have the equality 
\begin{equation}
\label{eq:2.4.1}
\begin{array}{cc}
\Omega(P)\ =\ \Omega\left (\frac{d^mP}{dt^m}\right )
\end{array}
\end{equation}
which is equivariant with the action of $\tau_\Omega$
\end{asse}
\begin{proof} It is sufficient to show the case $m\!=\!1$. We show a slightly
stronger statement:
{\it the subsequence $\{ X_{n_m}(P)\}_{m\in\Z_{\ge0}}$ converges to  
a series $a(s)$ if and only if  
 $\{ X_{n_m}\left(\frac{dP}{dt}\right)\}_{m\in\Z_{\ge0}}$ also converges to $a(s)$.}

For an increasing sequence
 $\{n_m\}_{m\in\Z_{\ge0}}$ and  for any fixed $k\!\in\!\Z_{\ge0}$, the convergence of the sequence 
$\frac{\gamma_{n_m-k}}{\gamma_{n_m}}$ to $c$ is
 equivalent to the convergence of the sequence 
$\frac{(n_m-k)\gamma_{n_m-k}}{n_m\gamma_{n_m}}\!=\!{\tiny (1\!-\!k/n_m)}\frac{\gamma_{n_m-k}}{\gamma_{n_m}}$
to the same $c$.
\end{proof}

\section{Finite rational accumulation}

We show that, if $\Omega(P)$ is a finite set, then it has a strong
structure, which we call the {\it finite rational accumulation} (\S3.2
Theorem and its Corollary). The whole sequel of the present paper focuses
on its study.  

\subsection{Finite rational accumulation}\hspace{0cm}
\label{subsec:3.1}

 We introduce the concept of {\it finite rational accumulation}. To this end,
we start with a preliminary concept: a {\it rational subset} of
 $\Z_{\ge0}$. The following fact is easy and well known, so we omit
 its proof.

\smallskip
\noindent
{\bf Fact.}\
{\it The following conditions for a subset $U\!\subset\!\Z_{\ge0}$ are
equivalent.}

i) {\it Put $U(t):=\sum_{n\in U}t^n$. Then, $U(t)$ is a rational function in
$t$. }

ii) {\it There exists $h\in\Z_{>0}$ and a polynomial $V(t)$ such that
$U(t)=\frac{V(t)}{1-t^h}$.}

iii) {\it There exists $h\in\Z_{>0}$ such that $n+h\in U$ iff $n\in U$ for
$n>\!>0$.}

iv) {\it There exists $h\!\in\!\Z_{>\!0}$, a subset $u\!\subset\!\Z/h\Z$ and a finite set 
$D\!\subset\! \Z_{\ge0}$ 
such that $U\!\setminus\! D\! =\! \cup_{[e]\in u}U^{[e]}\!\setminus\! D$, 
where, for a class  $[e]\! \in\!\Z/h\Z$ of $e$, put}
\begin{equation}
\label{eq:3.1.1}
 U^{[e]}\!:=\!\{n\!\in\!\Z_{\ge0}\mid n\! \equiv\! e \bmod h \}.
\end{equation}

Further more, ii), iii) and iv) are equivalent for a pair $(U,h)$. The least
such $h$ for a fixed $U$ will be called the {\it period} of $U$.

\begin{defn}
{\bf 1.} A subset $U$ of $\Z_{\ge0}$ is called a {\it rational subset} 
if it  satisfies one of the above four equivalent conditions.


{\bf 2.}\ A {\it finite rational partition} of $\Z_{\ge0}$
is a finite collection $\{U_a\}_{a\in \Omega}$ of rational 
subsets $U_a\!\subset\! \Z_{\ge0}$ 
indexed by a finite set $\Omega$ 
such that 
 there is a finite subset $D$ of $\Z_{\ge0}$ so that
one has the disjoint decomposition 
\[
  \Z_{\ge0}\setminus D=\amalg_{a\in \Omega}(U_a\setminus D). 
\]
In particular, for $h\in\Z_{>0}$, the partition $\mathcal{U}_h:=\{U^{[e]}\}_{[e]\in \Z/h\Z}$ of $\Z_{\ge0}$ is called the {\it standard partition of period} $h$.

{\bf 3.}\ For a finite rational partition $\{U_a\}_{a\in \Omega}$ of
 $\Z_{\ge0}$, the period of a standard partition, which subdivide
 $\{U_a\}_{a\in \Omega}$, is called a {\it period} of
$\{U_a\}_{a\in \Omega}$. The smallest period
 ($=\!\mathrm{lcm}\{\text{period of } U_a | \ a\in\Omega\}$)  of a finite rational partition 
$\{U_a\}_{a\in \Omega}$ is called  {\it the period} of $\{U_a\}_{a\in \Omega}$. 

\end{defn}

We, now, arrived at the key concept of the present paper.

\begin{defn}
A sequence $\{X_n\}_{n\in\Z_{\ge0}}$ of points in a Hausdorff space 
is {\it finite rationally accumulating} if the sequence
accumulates to a finite set, say $\Omega$, such that for a system
 of pairwise-disjoint open neighborhoods $\mathcal{V}_a$ 
for $a\!\in\!\Omega$,  
the system $\{U_a \}_{a\in\Omega}$ for 
$U_a\!:=\!\{n\!\in\!\Z_{\ge0}\mid X_n\!\in\!\mathcal{V}_a\}$ 
is a finite rational partition of $\Z_{\ge0}$. 
The (resp.\ a) period of the partition is called the (resp.\ a) {\it period of the finite rational accumulation set $\Omega$}.
 \end{defn}

\subsection{$\tau_\Omega$-periodic point in $\Omega(P)$}\hspace{0cm}

Generally speaking, finiteness of the accumulation set $\Omega$ of a sequence does
not imply that it is finite rationally accumulating (see \S3.5
Example a). Therefore, the following theorem describes a distinguished property of
the accumulation set $\Omega(P)$. This justifies the introduction of  the concept of 
``finite rational accumulation''. 

\begin{theo}
Let $P(t)$ be a tame power series in $t$. Suppose there exists an
 isolated point of $\Omega(P)$, say $a$, which is periodic with respect
 to the  $\tau_\Omega$-action on $\Omega(P)$. Then $\Omega(P)$ is  a finite rational 
accumulation set, whose period $h_P$ is equal to
 $\#\Omega(P)$. Furthermore, we have a natural bijection that identifies
 $\Omega(P)$ with the $\tau_\Omega$-orbit of $a$:
\begin{equation}
\label{eq:3.2.1} 
\begin{array}{rll}
\ \ \ \ \ \ \ \Z / h_P\Z\ \ & \simeq& \ \Omega(P) \\ 
 e \bmod h_P & \mapsto & a^{[e]}:=\underset{n\to\infty }{\lim}
  X_{e+h_P\cdot n}(P),  \!\!\!\!\!\!\!\!
\end{array}
\end{equation}
where the standard subdivision $\mathcal{U}_{h_P}$ of the partition of
 $\Z_{\ge0}$ is the exact partition for the space $\Omega(P)$ of the
 opposite series of $P$. The shift action $[e]\mapsto [e\!-\!1]$ in the LHS is
 equivariant to the $\tau_{\Omega}$ action in the RHS.  
\end{theo} 
\begin{proof}
The assumption on $a$ means: 

\noindent
i) There exists a positive integer $h\in\Z_{>0}$  such that 

\centerline{\qquad $(\tau_\Omega)^ha\!=\!a\!\not=\!(\tau_\Omega)^{h'}a$ \ \ for
 \ \  $0\!<\!h'\!<\!h$. }

\noindent
ii) There exists an open neighbourhood
 $\mathcal{V}_a$ of $a$ in $\C[[s]]$ such that 
\[
 \Omega(P) \cap \mathcal{V}_a\!=\!\{a\}.
\]
\noindent
In particular, $\Omega(P)\!\setminus\!\{a\}$ is a closed set.

\smallskip
Since $\Omega(P)$ is a compact Hausdorff space, it is a regular
 space, so  we may assume further that $\Omega(P)\cap
 \overline{\mathcal{V}_a}\!=\!\{a\}$.\!\!
Then, by setting
$U_a\!:=\!\{n\in\Z_{\ge0}\!\mid\! X_n(P)\!\in\! \mathcal{V}_a\}$, 
 the sequence $\{X_{n}(P)\}_{n\!\in\! U_a}$ converges to the unique
 limit element $a$. 
By the definition of $\tau_\Omega$ in \S2, the relation 
$(\tau_\Omega)^ha\!=\!a$ implies that 
the sequence $\{X_{n-h}(P)\}_{n\!\in\! U_a}$ converges to $a$. 
That is, there 
exists a positive number $N$ such that for any $n\!\in\! U_a$ with $n\!>\!N$, 
$X_{n\!-\!h}(P)\in \mathcal{V}_a$, and hence $n\!-\!h$ belongs to $U_a$. 

Consider the set $A\!:=\!\{[e]\!\in\! \Z/h\Z\mid $
there are infinitely many elements of $U_a$ which are congruent to 
$[e]$ modulo $h$ $\}$. By the defining property of $N$, if $[e]\in A$, then $U_a$ contains
 $U^{[e]}\cap \Z_{\ge N}$ ({\it Proof.} For any $m\in \Z_{\ge N}$ with $m\bmod
 h\equiv [e]$, there exists an integer $m'\in U_a$ such that $m'>m$ and $m'\bmod h=[e]$ by the
 definition of the set $A$. Then, by the definition of $N$, $m'-h\in
 U_a$. Obviously, either $m'-h=m$ or $m'-h>m$ occurs. If $m'-h>m$ then
 we repeat the same argument to $m''\!:=\!m'-h$ so that $m''-h=m'-2h\in U_a$. Repeating, similar
 steps, after finite $k$-steps, we show that $m'-kh=m\in U_a$). 

Thus, $U_a$ is, up to a finite number 
of elements, equal to the rational subset $\cup_{[e]\!\in\! A} U^{[e]}$. This implies 
$A\!\not=\!\emptyset$. Consider the rational subset $U_{(\tau_\Omega)^ia}:=\{n-i\mid n\in U_{a}\}$
 for $i=0,1,\cdots,h-1$. Due to \S2.3 Assertion 2, $\{X_{n}(P)\}_{n\!\in\!
 U_{(\tau_\Omega)^ia}}$ converges to $(\tau_\Omega)^ia$, so $U_{(\tau_\Omega)^ia}$  is, up to a finite number of elements, equal to the rational subset 
$\cup_{[e]\in A} U^{[e-i]}$.  By the assumption $a\not=\tau_\Omega^ia$
 for $0\le i<h$, any pair of rational subsets $U_{(\tau_\Omega)^ia}$
 ($0\!\le\! i<h$) have at most finite intersection, so $A$ is a singleton
 of the form $A\!=\!\{[e_0]\}$ for some $e_0\in\Z$ and $U_{(\tau_\Omega)^ia}\!=\!U^{[e_0-i]}$
 up to a finite number of elements. 
On the other hand, since the union 
$\cup_{i=0}^{h-1} U_{(\tau_\Omega)^ia}$ already covers $\Z_{\ge0}$ 
up to finite elements and since each $\{X_{n}(P)\}_{n\!\in\!
 U_{(\tau_\Omega)^ia}}$ converges only to $(\tau_\Omega)^ia$,
 the opposite sequence \eqref{eq:2.2.1} can have no other accumulating point than the set 
$\{a,\tau_\Omega a,\cdots,(\tau_\Omega)^{h-1}a\}$. 
%
That is,  $\Omega(P)$ is a finite rational accumulation set with the
 transitive $h_P$-periodic action of $\tau_\Omega$.
\end{proof}

\begin{coro} If the set of isolated points of $\Omega(P)$ is finite, then $\Omega(P)$ is a
 finite rational accumulation set with the presentation \eqref{eq:3.2.1}.
\end{coro}
\begin{proof} Since the $\tau_\Omega$ action preserves the set of isolated points of $\Omega(P)$, there should exists a periodic point.
\end{proof} 


\subsection{Example by Mach\`i~\cite{M}}\hspace{0cm}

Let $\Gamma:=\Z/2\Z *\Z/3\Z\simeq \mathrm{PSL}(2,\Z)$ with the
 generator system 
$G:=\{a,b^{\pm1}\}$ where $a,b$ are the generators of $\Z/2\Z$ and $\Z/3\Z$, 
respectively.
Then, the number $\#\Gamma_n$ of elements of
 $\Gamma$ expressed by the words in the letters $G$ of length less or
 equal than $n$ for $n\in\Z_{\ge0}$ is given by 
\[
\#\Gamma_{2k}= 7\cdot2^k-6   \text{\quad and\quad  }
\#\Gamma_{2k+1} = 10\cdot2^k-6  \text{\quad  for } k\in\Z_{\ge0}.
\]
Therefore, we get the following expression of the growth function:
\[\begin{array}{cccc}
P_{\Gamma,G}(t)&\ :=\ & \sum_{k=0}^\infty \#\Gamma_{k}t^{k} & =\ \frac{(1+t)(1+2t)}{(1-2t^2)(1-t)}.
\end{array}
\] 
Then, we see that $\Omega_1(P_{\Gamma,G})$   and, hence,  $\Omega(P_{\Gamma,G})$ are finite
 rationally accumulating of period 2. Explicitly, they are given as follows.

\vspace{0.2cm}

\centerline{
\quad\quad\ $\Omega_1(P_{\Gamma,G})\! =\! \Big\{
a_1^{[0]}\!:=\!\underset{n\to
 \infty}{\lim}\!\frac{\#\Gamma_{2n-1}}{\#\Gamma_{2n}}\!=\!\frac{5}{7}, \  
a^{[1]}_1\!:=\!\underset{n\to \infty}{\lim}\!\frac{\#\Gamma_{2n}}{\#\Gamma_{2n+1}}\!=\!\frac{7}{10}\Big\}$
}


\centerline{
 $\Omega(P_{\Gamma,G})= \Big\{\ a^{[0]}(s)\ ,\ a^{[1]}(s)\ \Big\}$\qquad \qquad \qquad \qquad \quad \
}
\noindent
where 
{
\[
 a^{[0]}(s): =\sum_{k=0}^\infty
 2^{-k}s^{2k}+\frac{5}{7}s\sum_{k=0}^\infty 2^{-k}s^{2k}\quad \quad \quad \quad \quad \quad 
\] }
{\large
\[
\quad
\begin{array}{lll}
& = 
 \frac{(1+\frac{5}{7}s)}{(1-\frac{s^2}{2})}
  =\frac{1}{2}\cdot \frac{1+\frac{5}{7}\sqrt{2}}{1-\frac{s}{\sqrt{2}}}+\frac{1}{2}\cdot\frac{1-\frac{5}{7}\sqrt{2}}{1+\frac{s}{\sqrt{2}}},
\end{array}
\] 
}
{
\[
 a^{[1]}(s): =\sum_{k=0}^\infty
 2^{-k}s^{2k}+\frac{7}{10}s\sum_{k=0}^\infty 2^{-k}s^{2k}
 \quad\quad\quad\quad\quad \ \ 
\]}
{\large
\[
\quad
 \begin{array}{lll}
& = 
\frac{(1+\frac{7}{10}s)}{(1-\frac{s^2}{2})}
 = \frac{1}{2}\cdot\frac{1+\frac{7}{5}\frac{1}{\sqrt{2}}}{1-\frac{s}{\sqrt{2}}}+\frac{1}{2}\cdot\frac{1-\frac{7}{5}\frac{1}{\sqrt{2}}}{1+\frac{s}{\sqrt{2}}}.
\end{array}
\]
}

\noindent
In \S5.4, these coefficients of fractional expansions are recovered by a
use of, so
called, rational operators (see \S5.3 Theorem ii)). 

We calculate also
$
\begin{array}{l}
r_P^2=R_P^2=a_1^{[0]}a_1^{[1]}= \frac{5}{7}\frac{7}{10}=\frac{1}{2}.
\end{array}
$
\subsection{Simply accumulating Examples} \hspace{0cm}

 A tame power series $P(t)$
is called {\it simply accumulating} if $\#\Omega(P)\!=\!1$. Growth
functions $P_{\Gamma,G}(t)$ for surface groups and Artin monoids are
simply accumulating, respectively (Cannon \cite{C},\! \cite{S2,S3}).\ This
fact for Artin monoids enables one to determine their F-functions \cite{S4}. 

\subsection{Miscellaneous Examples}\hspace{0cm}

Before going further, we use a simple model of oscillating sequence
$\{\gamma_n\}_{n\in\Z_{\ge0}}$ to give some examples of the power series $P(t)$ such that  

a) $\Omega_1(P)$ is finite but is not finite rationally accumulating, 

b) $\Omega_1(P)$ is finite rationally accumulating but
$\#\Omega_1(P)\!<\!\#\Omega(P)$, 

c) $\Omega(P)\not=\Omega(P+Q)$ for a power series $Q(t)$ for any $R_P>r_Q>r_P$.

\noindent
We do not use these results in the sequel so that the readers may skip
present subsection without substantial loss.

\smallskip
Given a triple $\U:=(U,a,b)$, where $U\!\subset\!\Z_{\ge1}$ is any 
infinite subset with infinite complement 
 and $a,b\in \C\setminus\{0\}$, we associate a sequence 
 $\{\gamma_n\}_{n\in\Z_{\ge0}}$ defined  by an induction on $n$: 
$ \gamma_0:=1 $ and 
$ \gamma_n:=
\gamma_{n-1}\cdot a
$ if $n\!\in\! U$ and $\gamma_{n-1}\cdot b$ if $n\!\not \in\! U$.
Set $P_\U(t):=\sum_{n=0}^\infty \gamma_nt^n$.
Then:

\medskip
\noindent
{\bf Fact i)} {\it The series $P_\U(t)$ is tame and $\Omega_1(P_\U)=\{a^{-1},b^{-1}\}$.}

 {\bf ii)} {\it The series $P_\U(t)$ is finite rationally accumulating if and
 only if $U$ is 
 a rational subset of $\Z_{\ge0}$.}

\smallskip
\noindent
{\it Proof.}  i) The inequalities: $\min\{|a|,|b|\}\le |\gamma_n/\gamma_{n-1}|
 \le \max\{|a|,|b|\}$ imply the tameness of $P_\U$. The latter half is
 trivial since the proportion $\gamma_{n}/\gamma_{n-1}$ takes only
the  values $a$ or $b$.

ii) This follows from: $P_\U$ is rational $\Leftrightarrow$ The sets $\{n\!\in\!\Z_{\ge1}\mid
 \gamma_n/\gamma_{n-1}\!=\!a\}\!=\!U$  and $\{n\!\in\!\Z_{\ge1}\mid
 \gamma_n/\gamma_{n-1}\!=\!b\}\!=\!U^c$ are rational$\Leftrightarrow$ $U$
 is rational.  $\Box$

\medskip
a)  By choosing a non-rational subset $U$, we obtain an example a). 

b)  Even if $U$
(and, hence, $U^c$ also) is a rational subset, if $\{U,U^c\}$ is not the standard partition of
 $\Z_{\ge0}$ of period 2, then the period of the partition  $\{U,U^c\}\! =\!
 \#\Omega(P_\U)\!>\!2\!=\!\#\Omega_1(P_\U)$. This gives an example  b). 

\smallskip
c)  To get an example satisfying c), we need a bit more consideration.
Define 
$ 
p_U:=\underset{n\to\infty}{\overline{\lim}}\frac{\#(U\cap [1,n])}{n} 
\text{\ and \ }
q_U:=\underset{n\to\infty}{\underline{\lim}}\frac{\#(U\cap [1,n])}{n} 
$.
If $U$ is a rational subset, then $p_U=q_U$ is a rational number. In
 general, the pair $(p_U,q_U)$ can be any of $\{(p,q)\in
 [0,1]^2\mid p\!\ge\! q\}$. 
Suppose $|a|\!\ge\!|b|$.
\[\begin{array}{lll}
 1/r_P :=\underset{n\to\infty}{\overline{\lim}}\ |a|^{\frac{\#(U\cap [1,n])}{n}}\cdot |b|^{1-\frac{\#(U\cap [1,n])}{n}} 
= |a|^{p_U}|b|^{1-p_U},\\
 1/R_P :=\underset{n\to\infty}{\underline{\lim}}\ |a|^{\frac{\#(U\cap [1,n])}{n}}\cdot |b|^{1-\frac{\#(U\cap [1,n])}{n}} 
= |a|^{q_U}|b|^{1-q_U}.
\end{array}
\]
Thus, $r_P$ and $R_P$ can take any values, satisfying:
$
|a|^{-1}\!\! \le\! r_P\! \le\! R_P\! \le\! |b|^{-1}.
$
If there is a gap $r_P\!<\!R_P$, then for any  $r\!\in\!\R_{>0}$  
such that $r_P\!<\! r\! <\! R_P$,  $Q(t)\!:=\!\sum_{n=0}^\infty
e^{i\theta_n}(t/r)^n$ for 
$\theta_n$ 
{\tiny $=\#(U\cap\Z_{1\le\cdot\le
n})\arg(a)\!+\!(n\!-\!\#(U\cap\Z_{1\le\cdot\le n}))\arg(b)$}
gives example c) (since {\small $\Omega_1(P_\U)\!\cap\!\{z\!\in\!\C:
 |z|\!=\!r\}\!=\!\emptyset$}  and \S2.4 Assertion\!\! 4).


\section{Rational expression of opposite series}

From this section, we restrict our attention to a tame power
series having the finite rational accumulation set $\Omega(P)$.

\subsection{Rational expression}\hspace{0cm}

\label{subsec:4.1}
We show that opposite series become rational functions
of special form. 
 We start with a characterization of a finite rational accumulation.
\begin{asse} 
Let $P(t)$ be a tame power series in $t$. 
The set  $\Omega(P)$ is a finite rational accumulation set of period 
$h_P\!\in\!\Z_{\ge1}$ if and 
only if $\Omega_1(P)$ is so.\
We say $P$ is finite rationally accumulating of period $h_P$.
\end{asse}
\begin{proof} 
If $\Omega(P)$ is finite rationally accumulating, then, in particular, the
 sequence $\frac{\gamma_{n-1}}{\gamma_n}$ is finite rationally
 accumulating.  
To show the converse and to show the coincidence of the periods, assume that
$\{\gamma_{n-1}/\gamma_n\}_{n\in\Z_{\ge0}}$ 
accumulate finite rationally of period $h_1$. 
Then, for the standard subdivision $\mathcal{U}_{h_1}\!:=\!\{U^{[e]}\}_{[e]\in
 \Z/h_1\Z}$, the subsequence 
$\{\gamma_{n-1}/\gamma_n\}_{n\in U^{[e]}}$ for each $[e]\!\in\!\Z/h_1\Z$
 converges to some number, which we denote by $a_1^{[e]}\!\in\! \C$ . 

For any $k\!\in\! \Z_{\ge0}$ and sufficiently large (depending on $k$) $n$, one has 
\[
\frac{\gamma_{n-k}}{\gamma_n}\ =\
\frac{\gamma_{n-1}}{\gamma_n}\frac{\gamma_{n-2}}{\gamma_{n-1}}\cdots
\frac{\gamma_{n-k}}{\gamma_{n-k+1}}.
\]
For $n\!\in\!  U^{[e]}$ with  $[e]\!\in\! \Z/h_1\Z$, we see that the RHS converges to
$a_1^{[e]}a_1^{[e-1]}\ldots a_1^{[e-k+1]}$. 
Then, for $[e]\!\in\! \Z/h_1\Z$ and $k\in\! \Z\!_{\ge0}$, by putting
\begin{equation}
\label{eq:4.1.1}
a_k^{[e]}\ :=\ a_1^{[e]}a_1^{[e-1]}\ldots a_1^{[e-k+1]},
\end{equation}
the sequence
$\{X_n(P)\}_{n\in  U_{[e]}}$ 
converges to 
$a^{[e]}\!:=\!\sum_{k=0}^\infty a_k^{[e]}s^k$ 
with $a_1^{[e]}=\iota(a^{[e]})$ so that $\Omega(P)$ is finite rationally
 accumulating. Its period $h_P$ is a divisor of $h_1$, but it cannot be
 strictly smaller than $h_1$, since otherwise the sequence
 $\{\gamma_{n-1}/\gamma_n\}_{n\in\Z_{\ge0}}$ gets a period shorter than $h_1$.
\end{proof}
\begin{rema}
That the period of the finite rational accumulation of $\Omega_1(P)$ is
 equal to $h_P$ does not imply $\#\Omega_1(P)\!=\!h_P$. That is, the map 
 $a\!\in\!\Omega(P)\mapsto \iota(a)\!\in\!\Omega_1(P)$ is not necessarily injective (see \S3.5 Example b).
\end{rema}

\begin{asse}
Let  $P$ be finite rationally accumulating of period  
$h_P\in\Z_{\ge1}$. 
Then the opposite series $a^{[e]}=\sum_{k=0}^\infty a_k^{[e]} s^k$ 
in $\Omega(P)$ associated with 
the rational subset $U^{[e]}$ converges to a rational function   
\begin{equation}
\label{eq:4.1.2}
a^{[e]}(s)\ =\ \frac{A^{[e]}(s)}
{1-A_Ps^{h_P}},
\end{equation}
where the numerator $A^{[e]}(s)$ is a polynomial in $s$ of degree 
$h_P\!-\!1$:  
\begin{equation}
\label{eq:4.1.3}
 \begin{array}{rll}
\qquad A^{[e]}(s)\ := \
\sum_{j=0}^{h_P-1}\left(\prod_{i=1}^{j}a_1^{[e-i+1]}\right)s^{j} 
\end{array}
\end{equation}
and 
\begin{equation}
\begin{array}{rll}
 \label{eq:4.1.4}
\quad A_P:=
\prod_{i=0}^{h_P-1}a_1^{[i]}=a_{h_P}^{[0]}
=\cdots=a_{h_P}^{[h_P-1]}. \!\!\!\!\!\!\!\!\!\!\!\!
\end{array}
\end{equation}
We have a relation
\begin{equation}
\label{eq:4.1.5}
(r_P)^{h_P} \ =\ (R_P)^{h_P} \ =\ |A_P|,
\end{equation}
where $r_P$ is the radius of convergence of $P(t)$ and $R_P$ is given by
 \eqref{eq:2.1.4}.
\end{asse}
\begin{proof} Due to the $h_P$-periodicity of the sequence $a_1^{[e]}$
 ($e\in\Z$), formula \eqref{eq:4.1.1} implies the ``{\it
 semi-periodicity}''  with respect to the factor \eqref{eq:4.1.4}:
\[
\qquad\qquad  a_{mh_P+k}^{[e]}\!=\!(A_P)^ma_k^{[e]} \quad \text{ for }
 m\!\in\!\Z_{\ge0}, \  k\!=\!0,\!\cdots\!,h_P\!-\!1.\!\!\!\!\!\!\!\!\!\!\!\!
\]
This implies a factorization $a^{[e]}\!=\!A^{[e]}\cdot\sum_{m=0}^\infty\!
 (A_Ps^{h_P})^m$ and hence \eqref{eq:4.1.2}. 

\medskip
To show \eqref{eq:4.1.5}, it is sufficient to
show the existence of 
positive real constants $c_1$ and $c_2$ such that for any $k\in\Z_{\ge0}$
there exists $n(k)\in\Z_{\ge0}$ and for any integer $n\ge n(k)$, one has
$
c_1 r^k\le \big| \frac{\gamma_{n-k}}{\gamma_n}\big| \le c_2r^k
$.
 
\noindent
{\it Proof.}\!
We may choose $c_1,c_2\!\in\!\R_{\!>\!0}$ satisfying 
$c_1\!\!<\!\!\min\{\big|\frac{a_i^{[e]}}{r^i}\big| \!\mid\! [e]\!\!\in\!\Z/h\Z, i\!\in\!\Z\cap[0,h\!-\!1]\}$
and
$c_2\!>\!\max\{\big| \frac{a_i^{[e]}}{r^i}\big| \mid [e]\in\Z/h\Z, i\!\in\!\Z\cap[0,h\!-\!1]\}$.
 \qquad $\Box$

This completes a proof of Assertion 8.
\end{proof}

\begin{coro} Let $\Omega(P)$ be finite. For any power
 series $Q(t)$ of radius $r_Q$ of convergence larger
 than $r_P$, $P\!+\!Q$ is tame and $\Omega(P)\!=\!\Omega(P\!+\!Q)$.
\end{coro}

\subsection{Coefficient matrix $M_h$ of numerator polynomials}\hspace{0cm}

\label{subsec:4.2}
In this and the next section, we study the linearly dependent relations among the opposite series $a^{[e]}(s)$ 
for $[e]\!\in\! \Z/h_P\Z$.

For the purpose, let us consider the matrix
\begin{equation}
\begin{array}{l}
\label{eq:4.2.1}
M_h:=(\prod_{i=1}^fa_1^{[e-i+1]})_{e,f\in\{0,1,\cdots,h-1\}}
\end{array}
\end{equation}
of the coefficients of the numerator polynomials \eqref{eq:4.1.3}. 
Regarding $a_1^{[0]},\!\cdots\!,a_1^{[h-1]}$ as variables, let us
introduce the ``discriminant'' by
\begin{equation}
\label{eq:4.2.2}
\quad  D_h(a_1^{[0]},\cdots,a_1^{[h-1]}) := \det(M_h)
 \ \ \in\ \Z[a_1^{[0]},\cdots,a_1^{[h-1]}].
\end{equation}
Actually, $D_h$ is an irreducible homogeneous polynomial  
of degree $h(h\!-\!1)/2$.
Under the cyclic permutation
$\sigma\!=\!(0,1,\cdots,h\!-\!1)$ of the variables,
\begin{equation}
\label{eq:4.2.3}
 D_h\!\circ\!\sigma\!=\!(-\!1)^{h\!-\!1}D_h.
\end{equation}
Our next task in \S4.3 is to stratify the zero-loci of $D_h$ according to the
rank of $M_h$. This is achieved  by introducing
the {\it opposite denominator polynomial} $\Delta^{op}$, whose degree
describes the rank of the matrix
$M_h$ (see \eqref{eq:4.3.3}). Here the coefficient is an arbitrary
field $K$. In particular, for the case of $K=\R$, we give a precise 
stratification of the positive real parameter space $(\R_{>0})^h$ of the
parameter $(a_1^{[0]},\cdots,a_1^{[h-1]})$, whose strata are labeled by cyclotomic polynomials i.e.\ an integral factor
of $1-s^h$ which contains also the factor $1-s$ (see Assertion 9.iv). 

\subsection{Linear dependence relations among opposite series}\hspace{0cm}
\label{subsec:4.3}
\begin{asse}
Fix $h\!\in\!\Z_{>\!0}$. For each $[e]\in \Z/h\Z$ and each $A\in
 K^\times$, let $A^{\![e]}(s)$ be the polynomial defined in equations \eqref{eq:4.1.3} and
 \eqref{eq:4.1.4} associated with any $h$-tuple 
$\bar{a}\!=\!(a_1^{[0]},\cdots\!,a_1^{[h-1]})\!\in\! (K^\times)^h$. 

{\rm i)} In $K[s]$, we have the equality of the greatest common  divisors:
\[
 \begin{array}{llllll}
&\gcd(A^{[0]}(s), 1\!-\!As^h)&=\cdots&=&\gcd(A^{[h-1]}(s), 1\!-\!As^h)\\
=\!&\gcd(A^{[0]}(s),A^{\![1]}(s))&=\cdots&=&\gcd(A^{[h-1]}(s),A^{\![h]}(s))
\end{array}
\]
\noindent
(whose constant term is normalized to 1), which we
 denote by $\delta_{\bar{a}}(s)$. 

Let us introduce the {\it opposite denominator polynomial} by
\begin{equation}
\label{eq:4.3.1}
\Delta_{\bar{a}}^{op}(s):= (1-As^h)/\delta_{\bar{a}}(s).
\end{equation}

{\rm ii)} For $[e]\!\in\! \Z/h\Z$, put
\begin{equation}
\label{eq:4.3.2}
b^{[e]}(s):= A^{[e]}(s)/\delta_{\bar{a}}(s). 
\end{equation}
The polynomials $b^{[e]}(s)$ for $[e]\!\in\! \Z/h\Z$ 
span the space $K[s]_{<\deg(\Delta_{\bar a}^{op})}$ of 
polynomials of degree less than $\deg (\Delta_{\bar{a}}^{op})$.
Hence, one has the equality:
\begin{equation}
\begin{array}{c}
\label{eq:4.3.3}
\rank\left(M_h\right)
\ =\ \deg (\Delta_{\bar{a}}^{op}).
\end{array}
\end{equation}

{\rm iii)} For $\varphi(s)\!\in\! K[s]$, $\varphi(s)\!\mid\! \Delta_{\bar{a}}^{op}$
 if and only if $\varphi(s)\!\mid\! 1\!-\!As^h$ and $\gcd(\varphi(s),
 A^{[e]}(s))$ $\!=\!1$. In
 particular, if $\bar{a}\!\in\!(\R_{>\!0})^h$, then
 $\Delta_{\bar{a}}^{op}$ is always divisible by  $1\!-\! ^h\!\!\!\!\sqrt{A} s$.

{\rm iv)}  Let $h\!\in\!\Z_{>0}$. There
 exists a stratification 
$\R_{>0}^h\!=\!\amalg_{\Delta^{op}}C_{\Delta^{op}}$, where the index set is
 equal to 
\begin{equation}
\label{eq:4.3.4}
\{\Delta^{op}\!\in\!\R[s] :  1\!-\!s \ |\ \Delta^{op}(s)\ |\
 1\!-\!s^h \ \& \ \Delta^{op}(0)\!=\!1\},
\end{equation}
 and
 $C_{\Delta^{op}}$\! is a smooth semi-algebraic set of $\R$-dimension $\deg(\Delta^{op})\!-\!1$, such that 
$\Delta_{\bar{a}}^{op}(s)\!=\!\Delta^{op}(^h\!\!\!\sqrt{A}s)$ for $\forall\bar{a}\!\in\!
 C_{\Delta^{op}}$ and $\overline{C_{\Delta_1^{op}}}\!\supset\!
 C_{\Delta_2^{op}} \Leftrightarrow \Delta_1^{op}|\Delta_2^{op}$
\end{asse}

\begin{proof} 
i) By Definitions \eqref{eq:4.1.3}, \eqref{eq:4.1.4} and \eqref{eq:4.1.1}, 
we have the following relations:
\begin{equation}
\label{eq:4.3.5}
a_1^{[e+1]}s A^{[e]}(s)+ (1-As^h)\ =\ A^{[e+1]}(s) 
\end{equation}
for  $[e]\!\in\! \Z/h\Z$.\ 
This implies
$\gcd(A^{[e]}(s),1\!-\!As^h)\!\mid\! \gcd(A^{[e+1]}(s),1\!-\!As^h)$ for 
$[e]\in \Z/h\Z$. Thus,  one may conclude that all of the polynomials

\noindent
$\gcd(A^{[e]}(s),1-As^h)$ $=\gcd(A^{[e]}(s),A^{[e+1]}(s))$ 
for $[e]\in \Z/h\Z$ are the same up to a constant factor.
It is obvious that a factor of $1-As^h$ contains a nontrivial constant
 term, which we shall normalize to 1.

ii) Let $V$ be the subspace of  $K[s]/(\Delta_{\bar{a}}^{op})$ spanned by
 the images of $b^{[e]}(s):=A^{[e]}(s)/\delta_{\bar{a}}(s)$ for $[e]\in \Z/h\Z$.
Relation \eqref{eq:4.3.5} implies that $V$ is closed under  
multiplication by $s$. On the other hand,  $b^{[e]}(s)$ 
and $\Delta_{\bar{a}}^{op}$ are relatively prime, so they
generate 1 as a $K[s]$-module. That is, $V$ contains the class [1] of 1. 
Hence, $V=K[s]\cdot[1]=K[s]/(\Delta_{\bar{a}}^{op})$. 
Since
 $\deg(b^{[e]}(s))\!=\!h\!-\!1\!-\!\deg(\delta_{\bar{a}}(s))\!=\!\deg(\Delta_{\bar{a}}^{op})\!-\!1$,
 $V\cap K[s]\Delta_{\bar{a}}^{op}\!=\!0$.
This means that the polynomials 
$b^{[e]}(s)$ for $[e]\in \Z/h\Z$ span the space of polynomials 
of degree less than $\deg(\Delta_{\bar{a}}^{op})$. 
In particular, one has $\rank(M_h)\!=\!\rank_KV\!=\!\deg(\Delta_{\bar{a}}^{op})$. 

iii) The first half is a reformulation of the
 definition of  $\delta_{\bar a}$ and \eqref{eq:4.3.1}. We see that if $1-rs\not| \Delta_{\bar{a}}^{op})$ then  $1-rs\mid
 A^{[e]}(s)$ \eqref{eq:4.3.2} so $A^{[e]}(1/r)=0$.  This is
 impossible,  
since all coefficients of $A^{[e]}$ and $1/r$ are positive reals.

iv)  Let $\Delta^{op}$ be a polynomial as given in \eqref{eq:4.3.4} and put $d = \deg(\Delta^{op})$.
Consider the set $\overline C_{\Delta^{op}}\!:=\!\{c(s)\!=\!1\!+\!c_1s\!+\!\cdots\!+\!c_{d-1}s^d\!\in\!
 \R[s]\mid \exists r\!\in\!\R_{>0} \ \text{s.t.\ all coefficients of } 
A_c^{[0]}\!:=\!c(s)(1\!-\!r^hs^h)/\Delta^{op}(rs)
 \text{ are positive}\}
 $. Then $\overline C_{\Delta^{op}}$ is an open semi-algebraic set in
 $\R^d$, 
 which is nonempty since $\Delta^{op}(rs)/ (1\!-\!rs)$
 belongs to $\overline C_{\Delta^{op}}$. 
In particular, it is pure dimensional of real dimension $d-1$.
To any $c\!\in \! \overline C_{\Delta^{op}}$, one can associate 
a unique $\bar{a}\!\in\!\!(\R_{>0}\!)^h$ such that the associated polynomial
 $A^{[0]}$ \eqref{eq:4.1.3} is equal to $A_{c}^{[0]}$. We identify $\overline C_{\Delta^{op}}$ with 
the semi-algebraic subset $\{a\!\in\! (\R_{>0})^h\!\mid\!
 a\!\leftrightarrow\! c\!\!\in\!\! \overline C_{\Delta^{op}}\}$ of pure
 dimension $d-1$ embedded in $(\R_{>0})^h$. Similarly, for any factor $\Delta'$ of $\Delta^{op}$
 (over $\R$) divisible by $1-s$, we consider the semi-algebraic subsets
 $\overline C_{\Delta'}$ in $\R_{>0}^h$ of pure dimension $\deg(\Delta')$. 
Then, the multiplication of $\Delta^{op}/
\Delta'$ induces the inclusion $\overline C_{\Delta'}\subset
 \overline C_{\Delta^{op}}$.
Then we define the semi-algebraic set
$C_{\Delta^{op}}$ inductively by 
$\overline C_{\Delta^{op}}\!\setminus\! \cup_{\Delta'} C_{\Delta'}$, 
where the index $\Delta'$ runs over all factors of $\Delta^{op}$
which are not equal to $\Delta^{op}$ and are divisible by $1\!-\!rs$. 
By the induction hypothesis, $d\!-\!1>\dim_\R(C_{\Delta'})$ so that 
the difference $C_{\Delta^{op}}$ is a non-empty open semi-algebraic set with
pure real dimension  $d-1$.

This completes the proof of Assertion 9.
\end{proof}

Suppose $char(K)\not|\ h$, and let $\tilde K$ be the 
splitting field of $\Delta_{\bar{a}}^{op}$ with the decomposition 
$\Delta_{\bar{a}}^{op}\!=\!\prod_{i=1}^{d}(1\!-\!x_is)$ in $\tilde K$ 
for $d:=\deg(\Delta_{\bar{a}}^{op})$. 
Then, one has the partial fraction decomposition:
\begin{equation}
\label{eq:4.3.6}
\begin{array}{l}
\frac{A^{[e]}(s)}{1-As^h}\ =\ \sum_{i=1}^{d}\frac{\mu^{[e]}_{x_i}}{1-x_is}
\end{array}
\end{equation}
for $[e]\in \Z/h\Z$, where $\mu^{[e]}_{x_i}$ is a constant in $\tilde K$ 
given by the residue:
\begin{equation}
\label{eq:4.3.7}
\begin{array}{l}
\qquad
\mu^{[e]}_{x_i}\ =\ \frac{A^{[e]}(s)(1-x_is)}{1-As^h} \Bigr|_{s=(x_i)^{-1}} 
\ =\ \frac{1}{h}A^{[e]}(x_i^{-1}) .
\end{array}
\end{equation}
\begin{coro}
The matrix $\big((\mu^{[e]}_{x_i})_{[e]\in\Z/h\Z,x_i^{-1}\in
 V(\Delta_{\bar a}^{op})}\big)$ is of maximal rank $d$.
\end{coro}
\begin{proof} The rational function on the LHS of \eqref{eq:4.3.6} for $[e]\!\in\!\Z\!/\!h\Z$ span a vector
 space of rank $d\!:=\!\deg(\!\Delta_{\bar{a}}^{op})$. Therefore, the coefficient
 matrix on the RHS has rank equal to $d$.
\end{proof}

\begin{rema}
1.  One has the equivariance 
$\sigma(\mu^{[e]}_{x_i})\!=\!\mu^{[e]}_{\sigma(x_i)}$ 
with respect to the action $\sigma\in {\rm Gal}(\tilde
 K,K)$ of the Galois group of the splitting field.

2. The index $x_i$ in \eqref{eq:4.3.7} may run over all roots $x$ of 
the equation $x^h\!-\!A\!=\!0$. However, if $x^{-1}\not\in
V(\Delta_{\bar{a}}^{op})$ 
(i.e.\ $\Delta_{\bar{a}}^{op}(x^{-1})\!\not=\!0$), then $\mu^{[e]}_{x}\!=\!0$.

3. For the given $h\in\Z_{>0}$, to consider the space of finite
 parameters $(a_1^{[0]},\cdots,a_1^{[h-1]})$ is equivalent to consider the
 space of infinite parameters $(a_i)_{i\in\Z}$ with ``quasi''-periodicity $a_{i+h}=Aa_i$. Then
it was suggested by the referee to regard the latter space over $\C$ as a
 $h$-''quasi''-periodic 
 representation of $\Z$ and to decompose it to the direct sum the
 sequence $(a_i=A^{i/h\chi(i))})$ for $\chi\in 
Z/h\Z\to \C^\times$.
\end{rema}

\subsection{The module $\C\Omega(P)$}\hspace{0cm} 

We return to a tame power series $P(t)$ \eqref{eq:2.1.1}.
Suppose  $P(t)$ is finite rationally accumulating  
of a period $h_P$. Let $a_1^{[e]}$ be the initial of the opposite 
series $a^{[e]}\!\in\! \Omega(P)$ for $[e]\!\in\! \Z/h_P\Z$. 
Since $\Delta_{\bar{a}}^{op}(s)$ \eqref{eq:4.3.1} for
$\bar{a}\!:=\!(a_1^{[0]},\cdots,a_1^{[h-1]})$ depends only on $P$ 
but not on the choice of a period $h_P$, we shall denote it  
by $\Delta_P^{op}(s)$ and call it the {\it opposite denominator polynomial} of
$P$.
Then, \S4.3 Assertion 9.ii) says that we have the $\C$-isomorphism:
\begin{equation}
\begin{array}{lll}
\label{eq:4.4.1}
\C\Omega(P)& \simeq& \C[s]/(\Delta_P^{op}(s)), \\
\quad a^{[e]}& \mapsto& b^{[e]}:=\Delta_P^{op}\cdot a^{[e]}\bmod\Delta_P^{op}.
\end{array}
\end{equation} 
Let us rewrite equality \eqref{eq:4.3.2} and introduce the key number:
\begin{equation}
d_P:=\rank_\C\big(\C\Omega(P)\big) = \deg(\Delta_P^{op}).
\end{equation}

Define an endomorphism $\sigma$ on $\C\Omega(P)$ by letting 
\begin{equation}
\begin{array}{l}
\sigma(a^{[e]})\!:=\!\tau_\Omega^{-1}(a^{[e]})\!=\!\frac{1}{a_1^{[e+1]}}a^{[e+1]}.
\end{array}
\end{equation}
\begin{asse} The actions of $\sigma$ on the LHS  and the 
multiplication of $s$ on the RHS of \eqref{eq:4.4.1} 
are naturally identified. Hence, the linear dependence relations 
among the generators $a^{[e]}$ ($[e]\!\in\!\Z/h\Z$) are 
obtained by the linear dependence relations
 $\Delta_{P}^{op}(\sigma)a^{[e]}$  for $[e]\!\in\!\Z/h\Z$.
\end{asse}
\begin{proof} The first part of Assertion 10 is a matter of calculation.

 $\sum_{[e]\!\in\!\Z/h\Z} c_{[e]}b^{[e]}\equiv
 0 \bmod \Delta_{P}^{op}(\sigma)b^{[e]}\!=\!0$  for $[e]\!\in\!\Z/h\Z$.
\end{proof} 

Note that the $\sigma$-action on $\C\Omega(P)$ is not $s|_{\C\Omega(P)}$ in the ring $\C[[s]]$.

\section{Duality theorem }

In this section, we restrict the class of functions $P(t)$ to those that
are analytically
continuable to a meromorphic function in a neighbourhood of the closed
disc of convergence.\footnote
{This assumption is necessary, since {\it the finite rational accumulation of
$P(t)$ does not imply that $P(t)$ is meromorphic on the boundary of its
convergent disc.} \\
 {\it Example.} 
Consider the function
$P(t):= \sqrt{\frac{1+t}{1-t}}=\sum_{n=0}^\infty \frac{(n-1)!}{2^n[n/2]![(n-1)/2]!}t^n$ which is
tame. We see that the sequence of the proportion
$\gamma_{n-1}/\gamma_n$ of its coefficients accumulates to the unique values 1, i.e.\
$\Omega_1(P)=\{1\}$ and $\Omega(P)=\{1/(1-s)\}$. On the other hand, we
observe that the function $P(t)$ has two singular points on the boundary
of the unit disc $D(0,1)$ which are not meromorphic but algebraic. Such
algebraic branching cases shall be treated in a forthcoming paper.
} 
Under this assumption, we show a duality
between $\Omega(P)$ and poles of $P(t)$ on the boundary of the disc.


\subsection{Functions of class $\C\{t\}_r$}\hspace{0cm}

For $r\in\R_{>0}$, we introduce  a class 
\begin{equation}
\label{eq:5.1.1}
\C\{t\}_r\!:=\!\Bigg\{
P(t)\!\in\!\C[[t]] \Big| 
\begin{subarray}{l}\text{
 i) $P(t)$ converges on the open disc $D(0,r)$.} \\
\text{ ii) $P(t)$ is analytically continuable to a meromorphic} \\
\text{\qquad function on an  open neighbourhood of $\overline{D(0,r)}$.
}
\end{subarray}
\!\Bigg\}\!\!\!\!\!\!\!\!\!\!\!
\end{equation}

For an element $P(t)$ of $\C\{t\}_r$, let us introduce a monic
polynomial $\Delta_P(t)$, called the {\it polar part polynomial} of
$P(t)$, characterized by 

i) \ $\Delta_P(t)P(t)$ is holomorphic in a neighbourhood of the circle
$|t|=r$,

ii) $\Delta_P(t)$ has lowest degree  among all polynomials satisfying i).

\medskip
\noindent
Next, we decompose
\begin{equation}
\label{eq:5.1.2}
\begin{array}{l}
\Delta_P(t)=\prod_{i=1}^N(t-x_i)^{d_i}
\end{array}
\end{equation} where 
 $x_i$ ($i\!=\!1\!,\!\cdots\!,\!N$, $N\!\in\!\Z_{\ge 0}$) 
are mutually distinct complex numbers 
with $|x_i|\!=\!r$ and $d_i\!\in\!\Z_{>0}$ ($i\!=\!1,\!\cdots\!,N$). 

\begin{defn} The {\it top denominator polynomial} $\Delta_P^{top}(t)$ of $P(t)$ is 
\begin{equation}
\begin{array}{l}
\label{eq:5.1.3}
\Delta_P^{top}(t)\!:=\!\prod_{i,d_i=d_{m}}\!\!(t-x_i) \quad\text{ where}\quad  
d_{m}\!:=\!\max\{d_i\}_{ i\!=\!1}^N. \!\!\!\!\!\!\!\!
\end{array}   
\end{equation}
\end{defn}

Note that $\Delta_P(t)$ may be equal to 1, and then
$\Delta_P^{top}(t)=1$. The converse: {\it if $\Delta_P(t)\not=1$, then
$\Delta_P^{top}(t)\not=1$}, is also true.

\subsection{The rational operator $T_U$}\hspace{0cm}

Associated with a
rational subset $U$ of $\Z_{\ge0}$, we introduce a linear operator $T_U$ acting on $\C\{t\}_r$ to itself, which we call a {\it rational
operator} or a {\it rational action} of $U$.
\begin{defn}
The action $T_U$ on $\C[[t]]$ of 
a rational subset $U$ of $\Z_{\ge0}$ is  
\begin{equation}
\label{eq:5.2.1}
\begin{array}{l}
T_U\ : \ P=\sum_{n\in \Z_{\ge0}}\gamma_nt^n
\quad \mapsto\quad   T_UP:=\sum_{n\in U}\gamma_nt^n. 
\end{array}
\end{equation}
One may regard $T_UP$ as a product of $P$ with the  
rational function $U(t)$ (\S3.1 Definition) in the sense of Hadamard \cite{H}. 
\end{defn}

The action $T_U$ is continuous w.r.t. the adic topology on
 $\C[[t]]$ since $T_U \big(t^k\C[[t]]\big)\subset t^k\C[[t]]$ for any
 $k\!\in\Z_{\ge0}$.  It is also clear that the radius of convergence of $T_UP$ is not less than that of
 $P$.

\begin{asse}
For $h\in \Z_{\ge0}$ and $[e]\in\Z/h\Z$, let us define the rational operator 
$T^{[e]}:= T_{U^{[e]}}$.
Then, we have 
 \vspace{-0.2cm}
{\footnotesize
 \begin{eqnarray}
\label{eq:5.2.2}
\sum_{e=0}^{h-1} T^{[e]} \ =\ 1, \qquad\qquad\qquad\qquad\qquad\\
\label{eq:5.2.3}
\qquad \qquad T^{[e]}\cdot t\ =\ t\cdot T^{[e-1]} \quad  \& \quad
T^{[e]}\cdot \frac{d}{dt}\ =\ \frac{d}{dt}\cdot T^{[e+1]}.
 \end{eqnarray}
}
\end{asse}
\begin{proof} 
The equation \eqref{eq:5.2.2} is a consequence of  $\Z_{\ge0}=\sqcup_{e=0}^{h-1}U^{[e]}$.
The \eqref{eq:5.2.3}: for any $t^m$
 ($m\in\Z_{\ge0}$), both sides return the same
 {\small $t^{m\!+\!1}\delta_{[e],[m\!+\!1]}\!=\!t^{m\!+\!1}\delta_{[e\!-\!1],[m]}$
 and
 $mt^{m\!-\!1}\delta_{[e],[m\!-\!1]}\!=\!mt^{m\!-\!1}\delta_{[e\!+\!1],[m]}$},
 respectively. 
\end{proof}

\begin{coro}
The action $T_U$ 
for a rational subset $U\subset\Z_{\ge0}$ 
preserves $\C\{t\}_r$ for any $r\!\in\!\R_{>0}$. The
highest order of poles on $|t|\!=\!r$ of $T_UP$ does not exceed that of
 $P\!\in\! \C\{t\}_r$.
\end{coro}
\begin{proof}
By decomposing the subset $U$ as in \S3.1 {\bf Fact} iv), we
need to prove this only for the case $U=U^{[e]}$ for some $[e\in\Z/h\Z]$
with $0\le e<h$. Since \eqref{eq:5.2.3} implies
$T^{[e]}\!=\!t^{e-h}T^{[0]}t^{h-e}$, we have only to prove the case when
$U=U^{[0]}=h\Z$. But, then, $T_{U^{[0]}}$, which maps $P(t)$ to
$\frac{1}{h}\sum_{\zeta}P(\zeta t)$, has the required property.
\end{proof}

\medskip
\subsection{Duality theorem}\label{subsec:5.3}\hspace{0cm}

The following is the goal of the present paper.

\bigskip
\begin{theo}\!\!\!\!\!{\bf (Duality)} \  
Let $P(t)$ be a tame power series belonging to $\C\{t\}_r$ for $r\!=\!r_P\ (=$
 the radius of convergence of $P${\rm )}. Suppose that $P(t)$   is finite
 rationally accumulating of period $h_P$. Then

\noindent
{\rm i)} The opposite denominator polynomial $\Delta_P^{op}(s)$ \eqref{eq:4.3.1}
 and
 the top denominator polynomial $\Delta_P^{top}(t)$ \eqref{eq:5.1.3} of $P(t)$ are
 opposite to each other. That is,
 \begin{equation}
\label{eq:5.3.1}
\deg_t(\Delta_P^{top}(t))\ =\ d_P \ =\ \deg_s(\Delta_P^{op}(s)),
\vspace{-0.2cm}
\end{equation}
and 
\vspace{-0.2cm}
\begin{equation}
\label{eq:5.3.2}
t^{d_P}\Delta_P^{op}(t^{-1})
 = \Delta_P^{top}(t), \text{\   equivalently \quad }
s^{d_P}\Delta_P^{top}(s^{-1}) = \Delta_P^{op}(s). 
 \end{equation}

\noindent
{\rm ii)}  We have an equality of transition matrices:
\begin{equation}
\label{eq:5.3.3}
\begin{array}{lll}
\!\Big( \frac{P(t)}{T^{[e]}P(t)}\big|_{t=x_i}\Big)_{[e]\in
 \Z/h_P\Z,\ x_i\in V(\Delta_P^{top}(t))} 
\!\!=
\Big( 
A^{[e]}\big|_{s\!=\!x_i^{-1}}\Big)_{[e]\in
\Z/h_P\Z,\ x_i^{-1}\in V(\Delta_P^{op}(s))} .
\end{array}
\end{equation}
In particular, $\Big( \frac{P(t)}{T^{[e]}P(t)}\big|_{t=x_i}\Big)_{[e]\in
 \Z/h_P\Z,\ x_i\in V(\Delta_P^{top}(t))}$ is of maximal rank $d_P$.
\end{theo}
\begin{proof}
We start with the following obvious remark.

\begin{asse} Let $c\in\C^\times$ be any non-zero complex
 constant. Change the variable $t$ to $\tilde t:=t/c$ and the opposite
 variable $s$ to $\tilde s:=cs$, and, for any tame series $P$, define a
 new tame series $\tilde P:=P|_{t=c\tilde t}$. 

Then we have, 

\smallskip
\centerline{
$\Omega(\tilde P)\ \ =\ \Omega(P)|_{s=\tilde s/c } \ :=\ \{a(\tilde s/c)\mid
 a(t)\in\Omega(P)\}$ , 
}
\centerline{
$\Omega_1(\tilde P)\ \ =\ \ \Omega_1(P)/c\ :=\ \{a_1/c\mid
 a_1\in\Omega_1(P)\}$.\quad\ \ 
}
\end{asse}
\begin{proof} The equalities follows immediately from direct calculations.
\end{proof}

According to Assertion 12, we prove the theorem by changing the
 variable $t$ to $\tilde t=t/c$ for $c=\ ^{h_P}\!\!\!\!\sqrt{A_P}$ (recall
 \eqref{eq:4.1.4}) so that the new tame series has the constant
 $A_{\tilde P}$ equal to 1. Therefore, from now on, in the present proof,
we shall assume that $P$  is a
 finite rationally accumulating tame series with $A_P\!=\!1$.
In particular, this implies that the radius $r_P$ of
 convergence of $P$ is equal to 1 (recall \eqref{eq:4.1.5}).

\medskip
\noindent
\ \ We first prove the theorem for a special but the key case when $\#\Omega(P)\!=\!1$.\!

\begin{asse} If $P(t)$ is simply accumulating
then $\Delta_P^{top}\!=\!t\!-\!1$.
\end{asse}
{\it Proof.} Consider the partial fractional expansion of $P$:
\begin{equation}
\begin{array}{lll}
\label{eq:5.3.4}
P(t)=\sum_{i=1}^N\sum_{j=1}^{d_i} 
\frac{c_{i,j}}{(t-x_i)^j} + Q(t),
\end{array}
\end{equation}
where $x_i$ ($i=1,\cdot,N$) is the location of a pole of $P$ of order
 $d_i$ on the unit circle $|x_i|=1$, $c_{i,j}$ ($j=1,\cdots,d_i$) is a
 constant in $\C$, and $Q(t)$ is a holomorphic function on a disc of
 radius $>1$. 

We apply stability (Assertion 3 in \S2.5) to the partial
 fractional expansion \eqref{eq:5.3.4}, to obtain
 $\Omega(P)\!=\!\Omega(P\!-\!Q)$. That is, the principal part
 $P_0\!:=\!P\!-\!Q$ gives rise to a simply
 accumulating power series. That is, $X_n(P_0)\!=\!\sum_{k=0}^n
 \!\frac{\sum_{i=1}^N\sum_{1\le j\le d_m}c_{i,j}x_i^{k-n-1}(n-k;j)/(j-1)!}
{\sum_{i=1}^N\sum_{1\le j\le d_m}c_{i,j}x_i^{-n-1}(n;j)/(j-1)!} s^k$
 ($n\!=\!0,1,2,\cdots$) converges to 
$\frac{1}{1-s}\!=\!\sum_{k=0}^\infty s^k$. 
Then, under this assumption, we'll show that if $c_{i,d_{m}}\!\not=\!0$ then $x_i\!=\!1$. 


For each fixed $k\in\Z_{\ge0}$, the numerator and denominator of the
 coefficient of $s^k$ in $X_n(P_0)$ are polynomials in 
$n$ of degree $\le d_m$. 
Let $v_n\!:=\!\sum_{i\!=\!1}^Nc_{i,d_{m}}x_i^{-n-1}$ be 
the coefficients of the top-degree term $n^{d_m}/(d_m-1)!$ in the denominator. Since the range of
$v_n$ is bounded (i.e.\ $|v_n|\le\sum_i |c_{i,d_m}|$ due to the
 assumption $|x_i|=1$), the sequence for $n\!=\!0,1,2,\cdots$
 accumulates to a non-empty compact set in $\C$. 
 
First,  consider the case when the sequence $\{v_n\}_{n\in\Z_{ge0}}$ has a unique accumulating value $v_0$. 
Let us show that $v_0$ is non-zero and the result of Assertion 13 is true. ({\it Proof.} The mean sequence:
$\{(\sum_{n\!=\!0}^{M\!-\!1}v_n)/M\}_{M\!\in\!\Z_{\!>\!0}}$ 
also converges to 
$v_0\!=\!\underset{n\to\infty}\lim v_n$. This means that  
$\sum_{i\!=\!1}^N c_{i,d_{m}}\frac{\sum_{n\!=\!0}^{M-1}x_i^{-n-1}}{M}$ 
converges to $v_0$. If $x_i\!\not=\!1$, the mean sum
 $\frac{\sum_{n\!=\!0}^{M-1}x_i^{-n-1}}{M}\!=\!\frac{1-x_i^{-M}}{(x_i-1)M}$
 tends to $0$ as $M\!\to\!\infty $. That is, $v_0\!=\!c_{1,d_m}$, where we assume
 $x_1\!=\!1$  (even if, possibly $c_{1,d_m}\!=\!0$).
That is, the sequence $v_n'\!:=\!v_n-c_{1,d_m}\!=\!\sum_{i\!=\!2}^Nc_{i,d_m}x_i^{-n-1}$ converges to 0. 
For a fixed $n_0\!\in\! \Z_{>0}$, consider the relations:
$v'_{n_0\!+\!k}\!=\!\sum_{i\!=\!2}^N (c_{i,d_m}x_i^{-n_0}) x_i^{-k+1}$ for $k\!=\!1,\cdots,N\!-\!1$.
Regarding $c_{i,d_m}x_i^{-n_0}$ ($i\!=\!2,\!\cdots\!,N$) as the unknown,
 we can solve the linear equation for them, since the Vandermonde determinant for the 
matrix $(x_i^{-k+1})_{i\!=\!2,\!\cdots\!,N, k\!=\!1,\!\cdots\!,N\!-\!1}$
 does not vanish.  
So, we obtain a linear approximation: $|c_{i,d_m}|=|c_{i,d_m}x_i^{-n_0}|\le c\cdot \max\{|v'_{n_0\!+\!k}|\}_{k\!=\!1}^{N\!-\!1}$ ($i\!=\!2,\!\cdots\!,N$) 
for a constant $c>0$ which depends only on $x_i's$ and $N$ but not on $n_0$. 
The RHS tend to zero as $n_0\!\to\! \infty$, whereas the LHS are unchanged.
This implies $|c_{i,d_m}|\!=\!0$, i.e.\ $d_i\!<\!d_m$ for $i\!=\!2,\!\cdots\!,N$. As we have
 already remarked $\Delta_P(t)\not=1$ implies
 $\Delta_P^{top}(t):=\prod_{\substack d_i=d_m}(t-x_i)\!\not=\!1$,
 and hence $c_{1,d_m}$ cannot be 0. So $\Delta_P^{top}(t)=t-1$.

Next, consider the case when the sequence $v_n$ has more than two 
accumulating values. Then, one of them is non-zero. Suppose the subsequence $\{v_{n_m}\}_{m\in \Z_{>0}}$ 
converges to a non-zero value, say $c$. 
Recall the assumption that the sequence $\gamma_{n-1}/\gamma_n$ converges 
to 1. So, the subsequence 
$\frac{\gamma_{n_m-1}}{\gamma_{n_m}}=\frac{v_{n_m-1}+ \text{lower terms}}
{v_{n_m}+ \text{lower terms}}$ 
should also converge to 1 as $m\to \infty$. 
In the denominator, the first term tends to $c\!\not=\!0$ 
and the second term (= (a polynomial in $n$ of degree
 $d_m\!-\!1$)$/n^{d_m}$) tends to zero. Similarly, in the numerator, the second term 
tends to zero. This implies that the first term in the numerator also converges to $c\not=0$. Repeating the same argument, we see that for any $k\in\Z_{\ge0}$,
 the subsequence $\{v_{n_m-k}\}_{m\in\Z_{>\!\!>0}}$ converges to the same $c$.
Then, for each fixed $M\in \Z_{>0}$, 
the average sequence $\{(\sum_{k=0}^{M-1}v_{n_m-k})/M\}_{m\in \Z_{>\!\!>0}}$ 
converges to $c$, whereas the
 values is given by
 $\sum_{i=2}^Nc_{i,d_m}x_i^{-n_m}\frac{1-x_i^{-M}}{(1-x_i^{-1})M}+c_{1,d_m}$
 which is close to $c_{1,d_m}$ for sufficiently large $M$ and $n_m>\!\!>M$. This implies $c=c_{1,d_m}$. 
Thus, the sequences 
$\{v'_{n_m-k}=\sum_{i=2}^Nc_{i,d_m}x_i^{n_m-k}\}_{m\in\Z_{>\!\!>0}}$ for any $k\ge0$ converge to 0.
Then, an argument similar to that of the previous case implies 
$|c_{i,d_m}|\!=\!0$, i.e.\ $d_i<d_m$ ($i\!=\!2,\!\cdots\!,N$). Hence, we have  $\Delta_P^{top}(t)=t-1$.

The proof of Assertion 13 is complete. \qquad\qquad\qquad\qquad  $\Box$

\medskip
We return to the proof of the  general case, where
$P$ is finite rationally accumulating of period $h$, but may no longer be
 simply accumulating.

\begin{asse} Let $P\in\C\{t\}_1$ be finite rational accumulating and the
 top denominator polynomial of $P$ is defined as in \eqref{eq:5.1.3}. Then, 

i) The top denominator polynomial of $P$ is a factor of
 $t^h-1$. 

ii) For any $0\!\le\! f\!<\!h$, $T^{[f]}P$ as a power series in
 $\tau:=t^h$ is simply accumulating, where top order of its denominator is
 equal to $d_m$.
\end{asse}
\begin{proof}
Since $P$ is rationally finite accumulating of period $h$ with radius 
of convergence $r_P\!=\!1$,
 we have $\lim_{m\to\infty} \gamma_{f+(m-1)h}/\gamma_{f+mh}=1(=r_P^h)$ for any {\small $0\le\! f\!<\!h$}. 
Regarding   
$T^{[f]}P=t^f\!\sum_{m=0}^\infty\! \gamma_{f+mh}\tau^m$ 
as a power series in $\tau\!=\!t^h$ and $t^f$ as a
 constant factor of the series, this implies that
 $\Omega_1(T^{[f]}P)=\{1\}$ and, hence, that $T^{[f]}P$ is simply accumulating.   
Then, Assertion 13 implies that the highest order poles of 
$T^{[f]}P$ (in the variable $\tau$) is only at $\tau\!-\!1\!=\!0$ for
 all $[f]\in\Z/h\Z$, and Corollary to Assertion 11 implies that the order of the pole
 at $\tau=1$ is less or equal than $d_m:=$the highest order of poles of 
 $P(t)$. Thus, we get an expression 
$T^{[f]}P=t^f\frac{g^{[f]}(\tau)}{(\tau-1)^{d_m}}$, where
 $g^{[f]}\in\C\{\tau\}_1$ such that orders of poles of $g^{[f]}$ is
 strictly less than $d_m$.
In view of \eqref{eq:5.3.4}, we obtain

$*$) \centerline{
$P = \sum_{f=0}^{h-1}T^{[f]}P= \frac{\sum_{f=0}^{h-1}t^fg^{[f]}(\tau)}{(\tau-1)^{d_m}} $. \quad
}

\noindent
This means, in particular, the the location of poles of $P$ of top order
 $d_f$ is contained in the solutions of $t^h-1=0$, i.e.\
 $\Delta^{top}(t)|(t^h\!-\!1$) and i) is proven. To show the the latter
 half of ii), we need to show that $g^{[f]}(1)\not=0$ for all
 $f$. However, $*$) says that 
$g^{[f_0]}(1)\not=0$ for some $f_0$. 

Assuming $g^{[f]}(1)=0$ for some $f$, we show a contradiction. Consider the
 sequence $\{\gamma_{f+mh}/\gamma_{f_0+mh}\}_{m\in\Z_{\ge0}}$. On one
 side, this converges to a non-zero number since $P$
 is finite rational accumulating of order $h$. On the other hand, since
 $g^{[f_0]}(\tau)/(\tau\!-\!1)^{d_m}$ has pole of order $d_m$ only at
 $\tau\!=\!1$, we have
 $\gamma_{f_0+mh}\!=\!g^{[f_0]}(1)m^{d_m}\!+\!O(m^{d_m-1})$ and order of poles of $g^{[f]}(\tau)/(\tau\!-\!1)^{d_m}$ are strictly less
 than $d_m$ by assumption, 
we have $\gamma_{f+mh}\!=\!O(m^{d_m-1})$. Thus the sequence converges to
 0, which contradicts to the non-zero limit! 
\end{proof}

For $0\!\le\! e,f\!<\!h$, let us calculate 
the value of the proportion $\frac{T^{[f]}P}{T^{[e]}P}(t)$ at a root $x$
 of the equation $t^h\!-\!1$ (defined by cancelling the poles at the
 point as a meromorphic function). 
\[
\!\!\!\!\!\!\!\!\!\!\!\!*)\qquad\qquad\qquad\quad
\frac{T^{[f]}P}{T^{[e]}P}(t)\biggr|_{t=x}\ =\ x^{f-e}\ \frac{g^{[f]}\big|_{\tau=1}}
{g^{[e]}\big|_{\tau=1}} .\qquad
\]
In order to calculate this value, we prepare an elementary Fact. 

\medskip
\noindent
{\bf Fact.} {\it Let $A(\tau)\!=\!\sum_{m=0}^\infty
 a_m\tau^m,B(\tau)\!=\!\sum_{m=0}^\infty b_m\tau^m\in\C\{\tau\}_1$ such
 that their highest order poles of the same order $d$ exist only at $\tau=1$.  Then,

\noindent
$**)$
\centerline{
\large{
 $\frac{A(\tau)}{B(\tau)}\big|_{\tau\!=\!1}=\underset{m\to\infty}{\lim}\ \frac{a_m}{b_m}$.\quad}
}\vspace{-0.4cm}
}
\begin{proof} Replacing $t$ and $c_{ij}$ in (5.3.4) with $\tau$ and
 $a_{ij}$ or $b_{ij}$, respectively, the RHS of $**)$ is written as
$\underset{m\to\infty}{\lim}\!\! \frac{\sum_{i=1}^N\sum_{ j\le d}a_{i,j}x_i^{-m-1}(m;j)/(j-1)!}
{\sum_{i=1}^N\sum_{ j\le d}b_{i,j}x_i^{-m-1}(m;j)/(j-1)!}$, where $x_i$
 is a complex number with $|x_i|\!=\!1$ and $x_1\!=\!1$. Since {\small
 $a_{1,d}\!=\!(\tau-1)^dA(\tau)\big|_{\tau\!=\!1}$} and {\small $b_{1,d}\!=\!(\tau-1)^dB(\tau)\big|_{\tau\!=\!1}$} are non-zero but $a_{i,d}=b_{i,d}=0$ for $i\not=1$, this is
 equal to $\underset{m\to\infty}{\lim}\!\! \frac{a_{1,d}(m;d)/(d-1)!+O(m^{d-1})}
 {b_{1,d}(m;d)/(d-1)!+O(m^{d-1})}=\frac{a_{1,d}}{b_{1,d}}=\frac{A(\tau)}{B(\tau)}\big|_{\tau\!=\!1}$. 
\end{proof}

\noindent
Applying this Fact, the RHS of  $*)$  is equal to 
$ x^{f-e}\! \lim\limits_{m\to\infty}\! \frac{\gamma_{f+mh}}{\gamma_{e+mh}}$.\!
Then, applying to this expression a similar argument for
 \eqref{eq:4.1.1}, we obtain:
{\small
\begin{equation}
\label{eq:5.3.5}
\frac{T^{[f]}P}{T^{[e]}P}(t)\biggr|_{t=x}
\ =\ 
\begin{cases} 
x^{f-e}/a_1^{[f]}a_1^{[f-1]}\cdots a_1^{[e+1]}  &\text{if $e< f$}\\
\qquad 1&\text{if $e= f$}\\
x^{f-e} a_1^{[e]} a_1^{[e-1]} \cdots a_1^{[f+1]}  &\text{if $e> f$}. 
\end{cases}
\end{equation}
}
Since the RHS are non-zero in all cases, the order of the poles of $T^{[e]}P(t)$ at a solution $x$
of the equation $t^h-1$ is independent of $[e]\in \Z/h\Z$. 
Summing up both sides of \eqref{eq:5.3.5} for $0\!\le f\!<\!h$, we obtain
{\small
\begin{equation}
\label{eq:5.3.6} 
{\small
\frac{P}{T^{[e]}P}(t)\biggr|_{t=x}
=\ A^{[e]}(x^{-1}).
}
\end{equation}
}
(recall the $A^{[e]}(s)$ \eqref{eq:4.1.3}).
Let $x$ be a solution of $t^h\!-\!r^h\!=\!0$ but 
$\Delta_P^{op}(x^{-1})\!\not=\!0$. Then 
$\delta_a(x^{-1})\!=\!0$ (see \eqref{eq:4.3.1}) and $A^{[e]}(x^{-1})\!=\!0$
for all $[e]\!\in\! \Z/h\Z$ (see Assertion 9. i)). 
That is, $\frac{T^{[e]}P}{P}(t)$ has a pole at $t\!=\!x$. 
This implies that $P(t)$ cannot have a pole of order $d_m$ at $t\!=\!x$
 (otherwise, due to Corollary to Assertion 11, the pole at $t\!=\!x$ of $T^{[e]}P$ is
 at most of order $d_m$, which is cancelled in  $\frac{T^{[e]}P}{P}(t)$
 by dividing by $P$, yielding a contradiction!). 
That is, we get one division relation.

\begin{asse}
$\Delta_P^{top}(t)\mid t^{d_P}\Delta_P^{op}(t^{-1})$ and
 $\deg(\Delta^{top}_P)\le d_P$.
\end{asse}

Finally, let us show the opposite division relation.

\begin{asse} Let $P(t)$ be a tame power series belonging to
 $\C\{t\}_r$, which is finite rationally accumulating of period $h$. Then

i) There exists a constant $c\!\in\!\R_{>0}$ such that $|\gamma_n|\!\ge\! cr^{-n}n^{d_m}$ 
for $n\!>\!\!>\!0$.

ii) 
\centerline{
$t^d \Delta_P^{op}(t^{-1}) \ | \ \Delta_P^{top}(t)$\ .\qquad\qquad
}
\end{asse}
{\it Proof.} 
i)  Consider the Taylor expansion of the partial fractional expansion {eq:5.3.4}. 
Using notation $v_n$ in Assertion 13, 
we have $\gamma_n=-v_n\frac{ r^{-n-1}(n;d_m)}{(d_m-1)!}$
$+\! \text{(terms coming from 
poles of order $<\!d_m$)}\! +\! \text{(terms coming from $Q(t)$)}$, 
where $v_n=\sum_{i}c_{i,d_m}(x_i/r)^{-n-1}$ depends only on $n \bmod h$ since 
$x_i$ is the root of the equation $t^h-r^h=0$. They cannot all be zero 
(otherwise, by solving the equations $v_n\!=\!0$ ($0\!\le\! n\!<\!h$), we
 get $c_{i,d_m}\!=\!0$ for all $i$, which contradicts to the vanishing of $d_m$). 
Let us show that none of the $v_n$ is zero. Suppose 
the contrary and $v_e\!=\!0\!\not=\!v_f$ for some integers $0\!\le\!
 e,f\!<\!h$. Then, one observes easily that 
$\underset{m\to\infty}{\lim}\frac{\gamma_{e+mh}}{\gamma_{f+mh}}=0$. This contradicts 
to formula \eqref{eq:5.3.5} and the non-vanishing of $a^{[e]}_1$ ($[e]\!\in\!\Z/h\Z$).

ii) Since $\Delta_P^{top}$ cancels all poles of maximal order, the fractional expansion of $\Delta_P^{top}(t)P(t)$ has 
poles of order at most $d_m\!-\!1$. Set
 $\Delta_P^{top}(t)\!=\!t^l\!+\!\al_1t^{l-1}\!+\!\cdots\!+\!\al_l$. Then,
 this means that the sequence $\{\gamma_N\}$ (Taylor coefficients of $P$) satisfies
\[
\gamma_{N}\cdot\al_l+\gamma_{N-1}\cdot\al_{l-1}+\cdots+\gamma_{N-l}\cdot1\ \ \sim\ \ o(N^{d_m}r^{-N})
\leqno{***)}
\]
as $N\! \to\! \infty$. 
Let $\sum_ka_ks^k\!\in\!\Omega(P)$ be an opposite series
given by a sequence $\{X_{n_m}(P)\}_{m\in\Z_\ge0}$ \eqref{eq:2.2.1}.
For each fixed $k\!\in\!\Z_{\ge l}$, substitute $N$ by $n_m\!-\!k\!+\!l$ in $***)$ and divide it by $\gamma_{n_m}$.
Then, taking the limit $m\!\to\infty$ using the part i), the RHS converges to 0, so that we get  
\[
a_{k-l}\al_l+a_{k-l+1}\al_{l-1}+\cdots+a_{k} \ =\ 0.
\]
Thus $s^l\Delta_P^{top}(s^{-1}) a(s)$ is a polynomial 
of degree 
$<\!l$ and  the denominator $\Delta_P^{op}(s)$ of $a(s)$ divides $s^l\Delta_P^{top}(s^{-1})$.
So, $d_P\!\le\! l$ and ii) is proved. 

This completes a proof of Assertion 16. \qquad \qquad  $\Box$

\smallskip
The proof of the theorem:\! \eqref{eq:5.3.1} and
 \eqref{eq:5.3.2} are already shown by Assertions 15 and 16, and
 \eqref{eq:5.3.3} is shown by \eqref{eq:4.3.7} and  \eqref{eq:5.3.6}.
\end{proof}

\subsection{Example by Mach\`i (continued)}\hspace{0cm}

Recall \S3.3 Mach\`i's example, where we learned that the
growth function $P_{\Gamma,G}(t)=\sum_{n=0}^\infty \#\Gamma_nt^n$  for the modular group
$\Gamma=\mathrm{PSL}(2,\Z)$ with respect to certain generator system $G$ is
equal to $\frac{(1+t)(1+2t)}{(1-2t^2)(1-t)}$ and that it is finite
rationally accumulating of period $h=2$. 

Using this data, we calculate further the rational actions on it.
\smallskip
\[
\begin{array}{l}
T^{[0]}P_{\Gamma,G}(t)=\ \ \ \sum_{k=0}^\infty \#\Gamma_{2k}t^{2k}\ \ \ =
 \frac{1+5t^2}{(1-2t^2)(1-t^2)}, \\
\\
T^{[1]}P_{\Gamma,G}(t)= \sum_{k=0}^\infty \#\Gamma_{2k+1}t^{2k+1} = \frac{2t(2+t^2)}{(1-2t^2)(1-t^2)},
\end{array}
\] 
The opposite denominator polynomial of the series $a^{[e]}$
($[e]\in\Z/2\Z$) and the top denominator polynomial of $P_{\Gamma,G}(t)$ are given as follows.
\[
 \Delta_{P_{\Gamma,G}}^{op}(s)=1-\frac{1}{2}s^2  \qquad \& \qquad \Delta_{P_{\Gamma,G}}^{top}(t)=t^2-\frac{1}{2}.
\]
Then the transformation matrix is given by
\[
\begin{array}{lll}
\!\!\!\!
\left[\!\!
\begin{array}{cc}
\frac{P_{\Gamma,G}(t)}{T^{[0]}P(t)}\!\!=\!\!
\frac{(1+t)^2(1+2t)}{1+5t^2}\!\mid_{t=\frac{1}{\sqrt{2}}}\!&\!
\frac{P_{\Gamma,G}(t)}{T^{[1]}P(t)}\!\!=\!\!\frac{(1+t)^2(1+2t)}{2t(2+t^2)}\!\mid_{t=\frac{1}{\sqrt{2}}}\!\\
\frac{P_{\Gamma,G}(t)}{T^{[0]}P(t)}\!\!=\!\!
\frac{(1+t)^2(1+2t)}{1+5t^2}\!\mid_{t=\frac{-1}{\sqrt{2}}} \!&\! 
\frac{P_{\Gamma,G}(t)}{T^{[1]}P(t)}\!\!=\!\!
\frac{(1+t)^2(1+2t)}{2t(2+t^2)}\!\mid_{t=\frac{-1}{\sqrt{2}}}\!
\end{array}\!\!
\right] 
\!=\!
{\footnotesize
\left[\!
\begin{array}{cc}
\!\!\footnotesize{1\!+\!
\frac{5}{7}\sqrt{2}} \!&\!
 \footnotesize{1\!+\!\frac{7}{5}\frac{1}{\sqrt{2}}}\! \\
\\
\!\!\footnotesize{1\!-\!\frac{5}{7}\sqrt{2}}\! &\! \footnotesize{1\!-\!\frac{7}{5}\frac{1}{\sqrt{2}}}\!
\end{array}\!
\right]
}\!.
\end{array}
\]
In fact, {\it this matrix coincides with the matrix}
$2\cdot\big(\mu^{[e]}_{x_i}\big)_{[e]\in\Z/2\Z,
x_i\in\{\pm\sqrt{2}^{-1}\}}$ \eqref{eq:4.3.7},
which was already calculated in \S3.3 Example as the coefficient of fractional expansion
of the opposite series $a^{[0]}$ and $a^{[1]}$. In particular, its
determinant, equal to $\frac{\sqrt{2}}{35}$, is non-zero. The matrix is an
essential ingredient of the trace formula for limit F-functions \cite[(11.5.6)]{S1}

\bigskip

\noindent
{\it Acknowledgement}: The author is grateful to Scott Carnahan
for his careful reading of the manuscript and suggesting corrections. He
is also grateful to the referee for the simplifications of the proofs
of Assertion 2.b. and Corollary to Assertion 11. 

\bigskip
\noindent
{\it Note of remembrance}:  The present paper was written shortly
after the earthquake and tsunami struck Tohoku aria of Japan on March
11, 2011. The author would
like to express his deep sorrow for the people who passed away in the
disaster.


\begin{thebibliography}{Sa}
\bibitem[C]{C}
Cannon, J.W.: The growth of the closed surface groups and the compact
	   hyperbolic Coxeter groups (unpublished, 1980's).

\bibitem[dH]{dH}
De la Harpe, P.: Topics in geometric group theory, the
	   university of Chicago Press, ISBN: 0-226-31719-6
	   (paper), ISBN 0-226-31721-8 (paper), p171-172.

\bibitem[E]{E}
Erd\"os, P.: Some remarks on the theory of graphs, Bull. A.A.S.\ {\bf 53}, p292-4.
\bibitem[H]{H}
Hadamard, J.: Th\'eoreme sur les s\'eries entieres, Acta math. {\bf 22} 
(1899), p55-63.


\bibitem[M]{M}
Mach\`i, Antonio: The growth of $\mathrm{PSL}(2,\Z)$, unpublished note (see \cite{dH}).

\bibitem[S1]{S1}
Saito, Kyoji:
Limit Elements in the Configuration Algebra for a Discrete monoid,
Publ. RIMS Kyoto Univ. {\bf 46} (2010), p31-113. DOI 10.2977/PRIMS/2

\bibitem[S2]{S2}
Saito, Kyoji:
Growth functions associated with Artin monoids of finite type,
	     Proc. Japan Acad. Ser. A {\bf 84} (2008), p179-183.  Zbl
	     1159.20330 MR 2483563

\bibitem[S3]{S3}
Saito, Kyoji:
Growth functions for Artin monoids,
	     Proc. Japan Acad. Ser. A {\bf 85} (2009), p84-88.  Zbl
	     pre05651160 MR 2548018
\bibitem[S4]{S4}
Saito, Kyoji:
Growth F-function for canncellative monoids, submitted.

\bibitem[SR]{SR}
Stanislaw Radziszowski: Small Ramsey Numbers, DS1, (2009), 72pp.
\end{thebibliography}
\end{document}